\documentclass[11pt]{article}
\usepackage{amsmath,amssymb,amscd,latexsym,amsthm,mathrsfs,amsfonts}
\usepackage[unicode]{hyperref}
\usepackage{hypbmsec,longtable}
\usepackage{graphicx}
\usepackage{fancyhdr}
\ifx\pdfoutput\undefined\else
\hypersetup{
colorlinks=true,
linkcolor=blue,
citecolor=blue,
urlcolor=blue,
filecolor=blue,
bookmarksnumbered=true,
pdfstartview=FitH,
pdfhighlight=/N
}
\fi
\usepackage{pgf,tikz}

\usepackage{array}
\allowdisplaybreaks[1]

\newtheorem{thm}{Theorem}[section] 

\newtheorem{lemma}[thm]{Lemma} 

\newtheorem{prop}[thm]{Proposition} 

\newtheorem{conj}[thm]{Conjecture} 

\theoremstyle{definition}

\newtheorem{defn}[thm]{Definition} 

\usepackage{cite}

\renewcommand{\author}[1]{\large\rm #1\\ \bigskip}
\newcommand{\address}[1]{{\normalsize\it #1\\}\bigskip}
\renewcommand{\title}[1]{\bigskip\bigskip\Large\bf #1\bigskip\bigskip\\}

\begin{document}

\begin{flushright}
\end{flushright}

\vglue .3 cm

\begin{center}

\title{Permutations sortable by deques and by  two stacks in parallel.}
\author{Andrew Elvey Price\footnote[1]{email: andrewelveyprice@gmail.com} and  Anthony J. Guttmann\footnote[2]{email:
                {\tt guttmann@unimelb.edu.au}}}
\address{ ARC Centre of Excellence for\\
Mathematics and Statistics of Complex Systems,\\
School of Mathematics and Statistics,\\
The University of Melbourne, Victoria 3010, Australia}

\end{center}
\setcounter{footnote}{0}

\begin{abstract}

Recently Albert and Bousquet-M\'elou \cite{AB15} obtained the solution to the long-standing problem of the number of permutations sortable by two stacks in parallel (tsip). Their solution was expressed in terms of functional equations. We show that the equally long-standing problem of the number of permutations sortable by a double-ended queue (deque) can be simply related to the solution of the same functional equations. Subject to plausible, but unproved, conditions, the radius of convergence of both generating functions is the same. Numerical work confirms this conjecture to 10 significant digits.
Further numerical work suggests that the coefficients of the deque generating function behave as $\kappa_d \cdot \mu^n \cdot n^{-3/2},$ where $\mu = 8.281402207\ldots,$ while the coefficients of the corresponding tsip generating function behave as $\kappa_p \cdot \mu^n \cdot n^{\gamma}$  with $\gamma \approx -2.473.$ The constants $\kappa_d$ and $\kappa_p$ are also estimated.

{\em Inter alia,}
 we study the asymptotics of quarter-plane loops, starting and ending at the origin, with weight $a$ given to north-west and east-south turns. The critical point varies continuously with $a,$ while the corresponding exponent variation is found to be continuous and monotonic for $a > -1/2,$ but discontinuous at $a=-1/2.$ 
\end{abstract}

  {


\vskip .2cm


\vskip .3cm

\section{Introduction}
\label{introduction}
The problem of {\em pattern-avoiding permutations} appears to have been first considered by MacMahon \cite{MM}. However Knuth \cite{K68} was the first to consider a number of classic data structures from the point of view of the permutations they could produce from the identity permutation (or, equinumerously, which permutations could produce the identity permutation). For the data structures considered - stacks and input-restricted deques - Knuth showed that of the $n!$ possible input permutations of length $n$ only $C_n \sim const \cdot 4^n \cdot n^{-3/2}$ and $S_{n} \sim const \cdot (3+2\sqrt{2})^{n} \cdot (n)^{-3/2}
$ could be sorted by stacks and input-restricted deques, respectively. This is a consequence of the fact that only 231-avoiding permutations are stack-sortable, as we discuss below.

Shortly thereafter, a number of authors, notably Evan and Itai \cite{EI71}, Pratt \cite{P73} and Tarjan \cite{T72} considered more general data structures. Foremost among these were two stacks in parallel, two stacks in series and deques. A deque, illustrated in Fig \ref{fig:deque}, is a double-ended queue, with insertions and deletions allowed at either end.  Until recently \cite{AB15}, none of these had been solved. As remarked above, simple stacks can sort any permutation that does not contain three successive (but not necessarily consecutive) elements in the order 231. We write this as the class $Av(231),$ that is, the class of $231$-avoiding permutations. For example 1573642 can't be sorted by a stack as the elements 562 (among other sub-sequences) are in the forbidden relative order. Input-restricted deques are describable by the class $Av(4231,\, 3241),$ whereas Pratt \cite{P73} showed that deques (without input or output restrictions) cannot be described similarly, as an infinite number of patterns would be needed in such a description.

To establish a notation, let $p_n$ denote the number of permutations of length $n$ that can be produced by two parallel stacks, let $d_n$ be the corresponding quantity for deques, and let $s_n$ be the corresponding quantity for two stacks in series. We name the corresponding generating functions $$P(t)=\sum p_n t^n,\,\; D(t)=\sum d_n t^n,\,\; {\rm and} \,\,\,S(t)=\sum s_n t^n.$$

In 2010, Albert, Atkinson and Linton \cite{AAL10} studied these problems with a view to establishing upper- and lower-bounds to the relevant growth constants. For deques they found $7.890 < \mu_d < 8.352$ and for tsips they found $7.535 < \mu_s < 8.3461,$ and commented that the actual growth constants may be equal, and may possibly be equal to exactly 8. As we show, the two growth constants do indeed appear to be equal, but to a slightly higher value, 8.28140...

In 2015 Albert and Bousquet-M\'elou \cite{AB15} found two coupled functional equations that give the generating function  $P(t).$ Unfortunately their representation does not allow for a single equation for the generating function, nor does it allow the asymptotics of $p_n$ to be obtained. However it does offer, in principle, a polynomial-time algorithm to obtain the coefficients $p_n.$ Other unanswered questions include the nature of the solution. Is it D-finite, or differentially algebraic? The answer to these questions is not known.

In their solution, Albert and Bousquet-M\'elou first encoded the operations involved in sorting a permutation as words over the alphabet $I_1,\,\,I_2,\,\,O_1,\,\,O_2,$ representing the input to, or output from, stack number 1 or stack number 2 respectively. Any sorting of a permutation can be effected by an operation sequence comprising a word in this alphabet, subject to certain constraints.
They then pointed out that such words could be considered from two other points of view. The first is a mapping to quarter-plane random walks that return to the origin. More precisely, these are random walks in ${\mathbb N} \times {\mathbb N}$ that start and end at the origin. Making the identification $I_1 \equiv N,$ $I_2 \equiv E,$ $O_1 \equiv S,$ $O_2 \equiv W,$ where $N,\,\,E,\,\,S,\,\,W,$ denote steps to the north, east, south and west respectively. Then operation sequences on words map to loops that return to the origin. The number of such  loops of length $2n$ is given by $C_n C_{n+1},$ where $C_n = \frac{1}{n+1} {2n \choose n}$ is the $n$th Catalan number.
However this is not a bijection as a given permutation may correspond to more than one loop.

An alternative representation was given in terms of two-coloured arches, in which the operation sequences were encoded as arches drawn between numbered points on a line, with arches above the line being of a different colour than those below the line. For further details of these connections the reader is referred to \cite{AB15}.

In this work our principal result is that the generating function for deques can be simply related to the generating function for two stacks in parallel. We therefore provide a comparable solution to the deque problem to that given in \cite{AB15} for the problem of two stacks in parallel in Theorem \ref{bigthm}, which states that the deque and tsip generating functions $D$ and $P,$ defined above, satisfy the following two (equivalent) equations:
\begin{equation}\label{SinRa}P(t)=\frac{(D(t)-1)\cdot (D(t)-t-1)}{2t\cdot (D(t)-1-t \cdot D(t))}\end{equation}
and
\begin{equation}\label{RinSa}2D(t)=2+t+ 2t P(t)-2t^2 P(t) -t\sqrt{1-4P(t)+4P(t)^2-8t P(t)^2+4t^2 P^2(t)-4t P(t)}.\end{equation}

 Subject to the veracity of conjectures 10, 11 and 12 in \cite{AB15}, we prove that the radii of convergence of the two generating functions $P(t)$ and $D(t)$ are the same. Further, with one extra assumption, in Theorem \ref{mediumthm} we prove that  $D(t)= D(t_c) + k_{D}\sqrt{t-t_{c}}+ o(\sqrt{t-t_{c}})$ for some constant $k_{D}\in\mathbb{R}$.

For both these problems we also give a detailed (numerical) study of the asymptotics, based on a 500 term expansion of the generating functions. In this way we estimate the radius of convergence to 10 significant digits for both problems, and find agreement at that level of precision, thus strengthening our confidence in the aforementioned conjecture that the critical points are identical. We also find the leading and next-to-leading critical exponents, which, perhaps surprisingly, are different for the two problems. For two stacks in parallel we find $$P(t) = P(t_c)-t_cP'(t_c)(1-t/t_c)+P_1(t_c)(1-t/t_c)^{\gamma_1} + P_2(t_c)(1-t/t_c)^{\gamma_2}  + o((1-t/t_c)^{\gamma_2}) $$
where $P_1(t_c), \,\, P_2(t_c)$ etc. are constants that occur as the lowest order terms in the expansion of implicitly defined functions $P_1(t), \,\, P_2(t)$ etc. around $t=t_c.$ 
We estimate
$t_c \approx 0.1207524976,$   
$\gamma_1\approx 1.473,$ and $ \gamma_2 \approx 1.946.$ 
Furthermore $$P(t_c)=\frac{1}{2(1-\sqrt{t_c})^2}.$$

For deques we find
$$D(t)  = D(t_c)+D_1(t_c)(1-t/t_c)^{1/2} + D_2(t_c)(1-t/t_c)^{\gamma_3}  +o(1-t/t_c)^{\gamma_3}, $$
where $ D_1(t_c), \,\, D_2(t_c)$ etc. are constants that occur as the lowest order terms in the expansion of implicitly defined functions $D_1(t), \,\, D_2(t)$ etc. around $t=t_c.$. We estimate
 $t_c \approx 0.1207524977,$ and $\gamma_3 \approx 0.973.$ Furthermore $$D(t_c)=\frac{1+t_c^{3/2}}{1-t_c},$$ and
\begin{equation}
D_1(t_c)=-2^{3/4}\cdot t_c^{3/2}\sqrt{P(t_c)^{3/2}+\sqrt{P(t_c)\cdot t_c}(1-\sqrt{t_c})\cdot P'(t_c)}.
\end{equation}

We have been unable to obtain a solution to the more difficult problem of the number of permutations of length $n$ sortable by two stacks in series, as shown in Fig. \ref{fig:2sis}, but have obtained the most extensive enumerations to date of the coefficients $s_n$ of the generating function $S(t),$ having obtained these exactly for $n < 20$ and approximately for longer permutations. An analysis of the asymptotics based on these enumerations is given in \cite{EG15}.

\begin{figure}[htbp]
   \centering
   \includegraphics[width=3.3in]{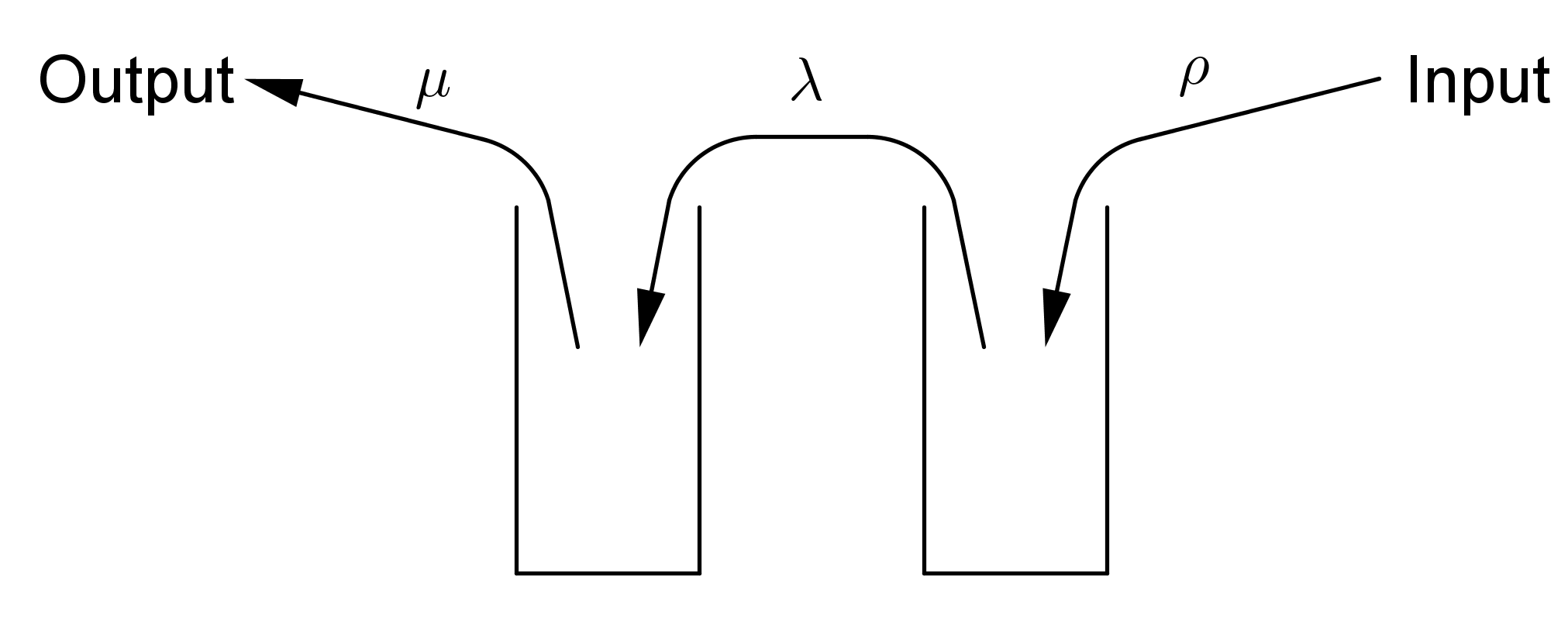} 
   \caption{Allowed input and output sequences for two stacks in series.}
   \label{fig:2sis}
\end{figure}

The associated problem of weighted quarter-plane loops also displays some unusual and hence interesting features for a simple lattice model. Recall that the loops can take steps in the $\pm x$ and $\pm y$ directions, and start and end at the origin $(0,0)$, and that all vertices must have non-negative co-ordinates. There is, in addition, a weight, or fugacity $a$ associated with NW and ES corners, so the generating function can be written $$Q(u,a) = \sum q_{n,m} u^n\cdot a^m$$ where $q_{n,m}$ is the number of such loops of $n$ steps with $m$ weighted corners. We are interested in the asymptotic behaviour of the coefficients, which is expected to be
$$[u^n]Q(u,a) \sim const \cdot q_c(a)^{-n} \cdot n^{g(a)}.$$
 As shown by Albert and Bousquet-M\'elou \cite{AB15},  not only does the critical point $q_c(a)$ vary with $a,$ which is to be expected, but so does the critical exponent $g(a),$ which is unexpected. In particular, this does not happen for similarly weighted loops in the half-plane or full-plane. Based on our numerical studies, we conjectured the variation of $q_c(a)$ with $a,$ providing numerical confirmation for the conjectured dependence in \cite{AB15}. We also studied the critical exponents, and conjectured these for various values of $a,$ but were not able to guess the $a$ dependence. In private correspondence, Kilian Raschel has given us his conjectured value for the exponents which is  
$$g(a) = \frac{\pi}{\arccos \left ( \frac{a-1}{a+1+\sqrt{2+a}} \right )}, \,\, {\rm for} \,\, a \ge 0.$$
This agrees with our numerical results, not only for $a \ge 0,$ but more broadly for $a > -1/2.$ However for $a=-1/2,$ we find the exponent takes the value $3/4,$ rather than $1$ as given by the above formula. Thus we have the additional interesting feature of a discontinuous critical exponent arising in this simple lattice walk model.

In the next section we establish our notation, define certain operations and prove a number of lemmas and propositions necessary for the proof of the main theorem. In the following section we derive the generating functions for various quantities, such as bi-coloured Dyck paths, weighted loops as defined above, and for a special class of sequences, defined below. Putting these definitions together leads us to a functional equation connecting $D(t)$ and $P(t),$ from which we obtain our principal result Theorem \ref{bigthm}. In section 4 we analyse this functional equation, and, subject to a conjecture in \cite{AB15}, prove Theorem \ref{mediumthm}, mentioned above. In section 5 we first discuss the asymptotics of weighted quarter-plane loops, then that of deques and two stacks in parallel. It is only through a remarkable cancellation that these two generating functions have different critical exponents. For deques we conjecture a square-root singularity, while for tsips the exponent arising in the generating function is approximately 1.473, which does not suggest any obvious rational number with low denominator. In a wild speculation we give a possible exact value for this exponent. The final section gives our conclusions.

\section{Notation and canonical operation sequences}

Consider the operations $I_{1},I_{2},O_{1},O_{2}$, where $I_{1}$ and $I_{2}$ represent input to the top and bottom of the deque, respectively, and $O_{1}$ and $O_{2}$ represent output from the top and bottom of the deque, respectively, as shown in Fig {\ref{fig:deque}. We will call a permutation {\em deque-achievable} if it can be produced by a deque, and we will call it  {\em tsip-achievable} if it can be produced by two stacks in parallel. Each achievable permutation can be encoded as a word over the alphabet $\{I_{1},I_{2},O_{1},O_{2}\}$, such that the total number of $I$s is the same as the total number of $O$'s, and any prefix contains at least as many $I$'s as $O$s. We call such a word an \emph{operation sequence}. For example, the permutation 4123 is achievable as it is produced by the operation sequence $I_{1}I_{1}I_{1}I_{2}O_{2}O_{2}O_{2}O_{2}$ (see Fig \ref{fig:deque4123}). This permutation is not, however, tsip-achievable.
\begin{figure}[htbp]
   \centering
   \includegraphics[width=3.3in]{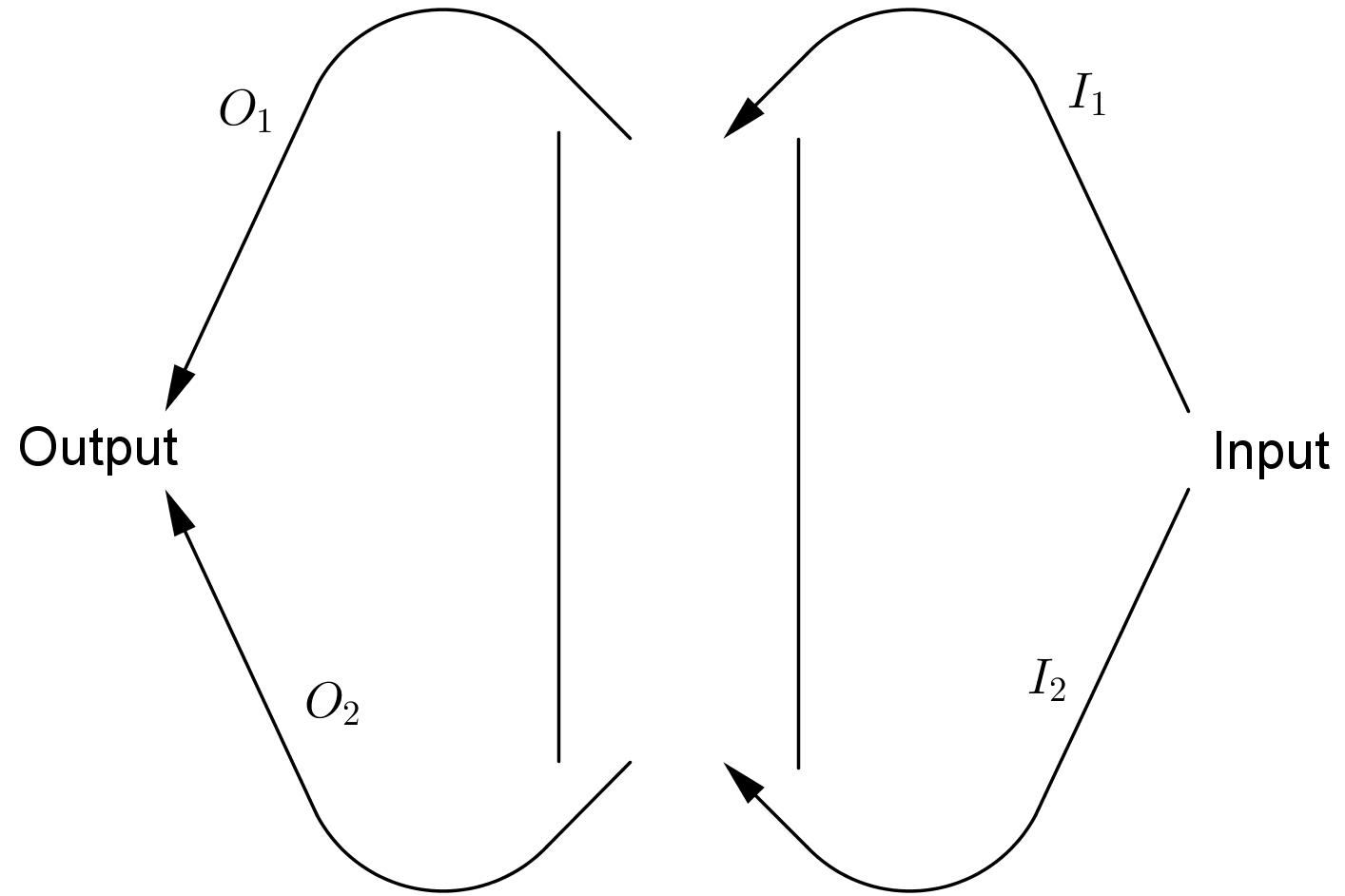} 
   \caption{The input and output operations $I_{1}$, $I_{2}$, $O_{1}$ and $O_{2}$ on a deque.}
   \label{fig:deque}
\end{figure}

\begin{figure}[htbp]
   \centering
   \includegraphics[width=4.3in]{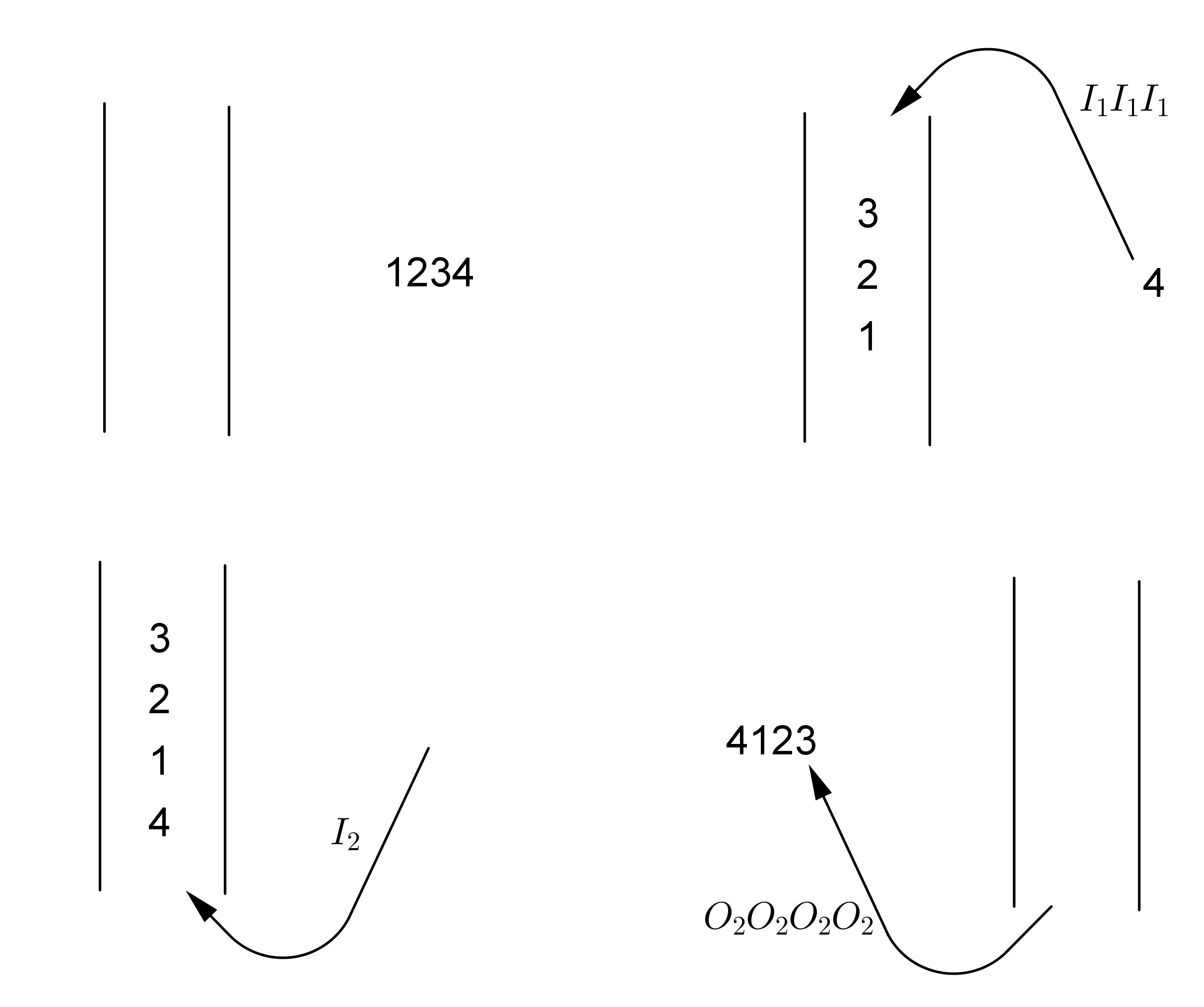} 
   \caption{The sorting procedure $I_{1}I_{1}I_{1}I_{2}O_{2}O_{2}O_{2}O_{2}$ produces the permutation 4123. This is the shortest permutation which cannot be produced by two stacks in series.}
   \label{fig:deque4123}
\end{figure}

We call an operation sequence $w$ a tsip word if it has the additional properties that for $j=1,2$ the total number of $I_{j}$'s is equal to the total number of $O_{j}$'s, and each prefix contains at least as many $I_{j}$'s as $O_{j}$'s. In this case, the operation sequence $w$ corresponds to a tsip-achievable permutation.  Note that if a permutation is tsip-achievable, then it is also deque-achievable, but the converse does not hold.

\begin{defn}We call an operation sequence $w$ {\em canonical} if it has the following three properties:
\begin{itemize}
	\item (outputs eagerly) $w$ contains no consecutive terms $I_{1}O_{2}$ or $I_{2}O_{1}$.
	\item (standard) Any tsip sub-word of $w$ begins with $I_{1}.$
	\item (top happy) Whenever one of the letters $I_{2}$ or $O_{2}$ appears in $w$, the number of preceding $I$s in $w$ must be at least two greater than the number of preceding $O$s. In other words, these two letters only appear when there are at least two numbers already in the deque.
\end{itemize}
\end{defn}

Note that this definition also defines what is meant by the terms {\em outputs eagerly,}  {\em standard} and {\em top happy.}

In Proposition \ref{dequeops} we will prove that every deque achievable permutation $\pi$ is produced by a unique canonical operation sequence. This is analogous to the result in \cite{AB15}, which in our language states that every tsip achievable permutation $\pi$ is produced by a unique standard tsip-word which outputs eagerly.

\begin{lemma} Any achievable permutation $\pi$ can be produced by a canonical operation sequence.\end{lemma}

\begin{proof} Assign the ordering $O_{1}<O_{2}<I_{1}<I_{2}$ to the letters. Since $\pi$ is achievable, it is produced by some operation sequence. Let $w$ be the operation sequence which produces $\pi$ and which is lexicographically minimal. Then we just need to show that $w$ is canonical.

If $w$ contains either $I_{1}O_{2}$ or $I_{2}O_{1}$ as a sub-word, then we can switch these letters to produce $w'<w$, where $w'$ also produces $\pi$. But this contradicts the minimality of $w$, so $w$ does not contain either $I_{1}O_{2}$ or $I_{2}O_{1}$. Hence $w$ outputs eagerly.

Suppose that $w$ contains a tsip sub-word $v$ which begins with $I_{2}$. Then we can form $v'$ by changing each subscript in $v$ to the other subscript. Note that $v'$ begins with $I_{1}$, so $v'<v$. Moreover, $v$ and $v'$ have the same effect on the state of the machine as each other. Hence, if $w=avb$, then the word $w'=av'b$ produces the same permutation $\pi$. But $w'<w$, which contradicts the minimality of $w$. So $w$ does not contain a tsip sub-word beginning with $I_{2}$. Hence, $w$ is standard.

Now suppose that $w$ can be split up as $w=uv$, where the number of $I$s in $u$ is at most one more than the number of $O$s, and where $v$ begins with either $I_{2}$ or $O_{2}$. Then we can form $v'$ by changing each subscript in $v$ to the other subscript. Note that $v'$ begins with $I_{1}$ if $v$ begins with $I_{2}$, and $v'$ begins with $O_{1}$ if $v$ begins with $O_{2}$, so in either case $v'<v$. After $u$ is applied to the machine, the deque contains at most one element, so $v'$ and $v$ have the same effect on the machine, with $v'$ essentially doing everything on the opposite end of the deque to $v$. Hence $uv'$ produces $\pi$. Moreover, $uv'<uv=w$, contradicting the minimality of $w$. Therefore, $w$ cannot be split as $uv$ in this way, so whenever one of the letters $I_{2}$ or $O_{2}$ appears in $w$, the number of preceding $I$s in $w$ must be at least two greater than the number of preceding $O$s. Hence, $w$ is top happy.

Therefore, $w$ is standard, top happy and outputs eagerly. Hence $w$ is canonical.\end{proof}

As in \cite{AB15}, we will consider the {\em type} of an operation sequence $w$:

\begin{defn} The {\em type} of an operation sequence $w$ is the word over the alphabet $\{I,O\}$ obtained by deleting all of the subscripts of $w$.\end{defn}

\begin{lemma}\label{type} If two operation sequences $u$ and $v$ produce the same permutation $\pi$ and both output eagerly, then $u$ and $v$ have the same type. Moreover, the $i$th operation of $u$ moves the same item as the $i$th  operation of $v$.\end{lemma}

\begin{proof}Let $u=u_{1}u_{2}\ldots u_{2n}$ and let $v=v_{1}v_{2}\ldots v_{2n}$. We will prove the lemma by induction. Assume that for some $k$, the words $u_{1}\ldots u_{k}$ and $v_{1}\ldots v_{k}$ have the same type. We will prove that the letters $u_{k+1}$ and $v_{k+1}$ also have the same type, and move the same item as each other. After the $k$th operation, both $u$ and $v$ have the same number of input steps and the same number of output steps as each other, so, since they produce the same sub-permutation $\pi$, the items currently in the deque according to $u$ must be the same as the items currently in the deque according to $v$, though the order may be different. Since $u$ and $v$ output eagerly, they will both output at this point if and only if the next item to be output is currently in the deque. Since this is the same for both $u$ and $v$, the letter $u_{k+1}$ and $v_{k+1}$ have the same type. Moreover, if they are both output steps, then the element for each will be the next element to be output, whereas if they are both input steps, the item moved will be the next element in the input. In both cases the item moved is the same for both $u$ and $v$. This completes the induction, so $u$ and $v$ have the same type.\end{proof}

\begin{prop}\label{dequeops}Every deque-achievable permutation $\pi$ is produced by a unique canonical operation sequence.\end{prop}

\begin{proof} We have already shown that every deque-achievable permutation is produced by some canonical operation sequence, so it remains to show that this operation sequence is unique. Let $u$ and $v$ be two canonical operation sequences that produce the permutation $\pi$, with $u\leq v$. Then we just need to show that $u=v$.

Suppose for the sake of contradiction that $u<v$. Let $u=u_{1}\ldots u_{2n}$ and let $v=v_{1}\ldots v_{2n}$, and let $k$ be minimal such that $u_{k}\neq v_{k}$, so $u_{k}<v_{k}$. Note then that before the $k$th move, the position of the deque, input and output is the same according to either $u$ or $v$. By Lemma \ref{type}, $u$ and $v$ have the same type, so the letters $u_{k}$ and $v_{k}$ have the same type. Therefore, $u_{k},v_{k}$ are either $O_{1}$, $O_{2}$, respectively, or $I_{1}$, $I_{2}$, respectively. Since $v_{k}$ is $O_{2}$ or $I_{2}$, and $v$ is top happy, there must be at least two items in the deque before this move. Now, by Lemma \ref{type}, the element moved by $u_{k}$ and $v_{k}$ must be the same, but $u_{k}$ and $v_{k}$ act at opposite ends of the deque, so they must both be input operations. Therefore, $u_{k}=I_{1}$ and $v_{k}=I_{2}$. We will now obtain a contradiction by showing that the operation sequence $v$ is not standard.

Define the sequences $k_{0},k_{1},\ldots,k_{m}$ and $l_{0},l_{1},\ldots,l_{m}$ of indices in $\{k,\ldots,2n\}$  as follows:
\begin{itemize}
\item $k_{0}=k$
\item For each $i$, the operations $v_{k_{i}}$ and $v_{l_{i}}$ move the same item, with $v_{k_{i}}$ an input and $v_{l_{i}}$ an output move. In particular, this determines $l_{0}$.
\item For $i\geq0$, the values $k_{i+1}$ and $l_{i+1}$ are chosen inductively so that $k_{i}<k_{i+1}<l_{i}<l_{i+1}$. Moreover, they are chosen to maximise $l_{i+1}$. If no such $k_{i+1}$ and $l_{i+1}$ exist, then we set $m=i$ and terminate. Note that the process must terminate at some point since $l_{1},l_{2},\ldots$ is a sequence of integers which is increasing and bounded above.
\end{itemize}
For $0\leq i\leq m$, let $t_{i}\in\{1,2\}$ be the value such that $v_{k_{i}}=I_{t_{i}}$ 

Let $D$ be the set of items which are in the deque just before operation $k$. Let $v_{j}$ be the first operation which outputs an item in $D$, and let this item be $d$. Then $d$ must be at one of the ends of the deque just before operation $k$. We will prove the following statements for $0\leq i\leq m$ in a single induction:
\begin{itemize}
\item $k\leq k_{i}<j.$
\item $u_{k_{i}}=I_{3-t_{i}}$. That is, $u_{k_{i}}$ and $v_{k_{i}}$ put the element they move at opposite ends of the deque.
\item $l_{i}<j.$
\item $v_{l_{i}}=O_{t_{i}}$. So $v_{k_{i}}$ and $v_{l_{i}}$ act on the same side of the deque.		\item $v_{l_{i}}=O_{t_{i}}.$
	\item $u_{l_{i}}=O_{3-t_{i}}$. So $u_{k_{i}}$ and $u_{l_{i}}$ act on the same side of the deque.
\item $u_{l_{i}}=O_{3-t_{i}}.$
\end{itemize} 

For the base case, by the definition of $k$ and $j$, we have $k_{0}=k<j$ and $u_{k_{0}}=u_{k}=I_{1}=I_{3-t_{0}}$.

Now for the inductive step. We will assume that $k\leq k_{i}<j$ and $u_{k_{i}}=I_{3-t_{i}}$ and show that $l_{i}<j$, $v_{l_{i}}=O_{t_{i}}$, $u_{l_{i}}=O_{3-t_{i}}$, $k\leq k_{i+1}<j$ and $u_{k_{i+1}}=I_{3-t_{i+1}}$. Let $a_{i}$ be the item moved by operations $u_{k_{i}},v_{k_{i}},u_{l_{i}}$ and $v_{l_{i}}$. Since $u_{k_{i}}=I_{3-t_{i}}$ and $v_{k_{i}}=I_{t_{i}}$, we have $\{u_{k_{i}},v_{k_{i}}\}=\{I_{1},I_{2}\}$. Therefore, one of $v_{k_{i}}$ and $u_{k_{i}}$ places $a_{i}$ on the end of the deque which $d$ has to be removed from. So $a_{i}$ must be removed before move $j$, hence $l_{i}<j$. Since there is an item $d$ on the deque throughout the whole time that $a_{i}$ is on the deque, the input and output operations which move $a_{i}$ must act on the same side of the deque. Hence, $v_{l_{i}}=O_{t_{i}}$ and $u_{l_{i}}=O_{3-t_{i}}$. Now to finish the induction, we just need to show that $k\leq k_{i+1}<j$ and $u_{k_{i+1}}=I_{3-t_{i+1}}$, in the case where $i<m$. First, $k\leq k_{i}<k_{i+1}<l_{i}<j$.

Also, by construction, the output move $v_{l_{i}}=O_{t_{i}}$ for $a_{i}$ occurs between the input and output moves for $a_{i+1}$, so we must have $v_{k_{i+1}}=I_{3-t_{i}}$, so $t_{i+1}=3-t_{i}$. Similarly, $u_{k_{i+1}}=I_{t_{i}}=I_{3-t_{i+1}}$. This completes the induction.

We will now show that $w=v_{k}v_{k+1}\ldots v_{l_{m}}$ is a tsip word. Since $l_{m}<j$, nothing is output from the deque from $D$.
Now, suppose some item is input but not output by $w$. Let the input and output moves of that item be $v_a$ and $v_b$, so $k<a<l_m<b$ and let $i$ be maximal such that $k_i<a$. If $i=m$ then $k_m<a<l_m<b$, which means that we shouldn't have terminated at $m$. If $i<m$ then $k_i<a<k_{i+1}<l_i<l_{i+1}<b$, which contradicts the maximality of $l_{i+1}$. Hence every item which is input by $w$ is also output by $w$. Therefore, $v_{k}\ldots v_{l_{m}}$ is a tsip sub-word of $v$ which begins with $I_{2}$, so $v$ is not standard. This contradicts the statement that $v$ is canonical, so we must have $u=v$. Therefore, any two canonical operation sequences which produce the same permutation are equal.
\end{proof}

\section{Enumeration}

In the previous section we showed that every deque-achievable permutation is produced by a unique canonical operation sequence, hence the number of deque achievable permutations of length $n$ is equal to the number of canonical operation sequences of length $2n$. Recall that this sequence is given by the generating function $D(t)$.

Before we proceed  any further, we will define the generating functions $Q$ and $P$ in the same way as in \cite{AB15}. Let $Q(a,u)$ be the generating function for quarter plane loops, counted by half-length, conjugate to the variable $u$ and number of NW or ES corners, conjugate to $a$. Let $P(t)$ be the length generating function for permutations which can be produced by two stacks in parallel. We note that this is the same as the generating function for standard tsip words which output eagerly, counted by half-length. Now using corollary 9 from \cite{AB15}, we have the following relation between these generating functions:


\begin{equation}\label{corS}Q\left(\frac{1}{P}-1,\frac{tP^{2}}{(1-2P)^{2}}\right)=2P-1.\end{equation}

Now we will define some new generating functions to characterise the length generating function $D(t)$ of deque-achievable permutations. We will call a non-empty operation sequence \textit{unbreakable} if it contains no tsip-sub-words, apart from possibly itself, and the deque is never empty during the associated procedure. Let $M(a,u,x)$ be the generating function for top happy, unbreakable operation sequences, counted by half-length, conjugate to $u$, number of appearances of a sub-word $I_{1}O_{2}$ or $I_{2}O_{1},$ conjugate to $a$, and number of times when there is only one element in the deque, conjugate to $x$.

We define a bi-coloured Dyck path to be a Dyck path where each step is coloured red or blue. A multicoloured peak is an up-step followed by a down-step of the opposite colour. Let $T(a,u,x)$ be the generating function for bi-coloured Dyck paths, where every step from height 0 or 1 is red, counted by half-length, conjugate to $u$, number of multicoloured peaks, conjugate to $a$, and number of vertices of the path at height 1, conjugate to $x$.

We define a new generating function $Q_{1}(a,u,x)$ for quarter plane loops where every step from the origin is a north step, counted by half-length, conjugate to $u$, number of NW or ES corners, conjugate to $a,$ and number of steps to the origin, conjugate to $x$. This is related to $Q$ by $Q(a,u)=Q_{1}(a,u,2)$. We can also think of $Q$ and $Q_{1}$ as generating functions for tsip words,  by replacing each $N$, $E$, $S$, $W$ step in the quarter plane loop with $I_{1}$, $I_{2}$, $O_{1}$, $O_{2}$, respectively. Then $Q(a,u)$ is the generating function for tsip operation sequences, counted by half-length, conjugate to $u$ and number of appearances of a sub-word $I_{1}O_{2}$ or $I_{2}O_{1}$, conjugate to $a$. The generating function $Q_{1}(a,u,x)$ is similar, except that it only counts those tsip words where every letter which appears after an equal number of $I$s and $O$s is an $I_{1}$. The other difference is that $Q_{1}$ also counts these words by the number of times there are an equal, non-zero number of $I$s and $O$s, conjugate to $x$.

\begin{lemma} The generating function for bi-coloured Dyck paths $T(a,u,x)$ is given by the equation
\begin{equation}\label{T}T(a,u,x)=\frac{4+2xu-2xau-x+x\sqrt{1-12u+4u^{2}-4au+4a^{2}u^{2}-8au^{2}}}{4-2xu-2xau-x+x\sqrt{1-12u+4u^{2}-4au+4a^{2}u^{2}-8au^{2}}}.\end{equation}\end{lemma}
\begin{proof}
Let $T_{1}(a,u,y)$ be the generating function for bi-coloured Dyck paths, where each step from height 0 is red, counted by half-length, conjugate to $u$, multicoloured peaks, conjugate to $a$, and  number of vertices of the path at height 0, including the end points, conjugate to $y$. Let $T_{2}(a,u)$ be the generating function for bi-coloured Dyck paths, counted by half-length, conjugate to $u$, and multicoloured peaks, conjugate to $a$. Note that each multicoloured Dyck path $\omega$ satisfies exactly one of the following:
\begin{itemize}
\item $\omega$ is empty,
\item $\omega$ is made up of a peak followed by a (possibly empty) multicoloured Dyck path,
\item $\omega$ is made up of an up-step, followed by a non-empty multicoloured Dyck path, followed by a down-step, followed by a (possibly empty) multicoloured Dyck path.
\end{itemize}
The contribution to  $T_{2}$ from the first case is simply $1$.

In the second case, the 4 possible peaks are counted by $u(2a+2)$, so the total contribution to $T_{2}$ from this case is $u(2a+2)T_{2}$.

In the third case, the contribution is simply $4u(T_{2}-1)T_{2}$. The multiplier 4 arises because there are four possibilities for the colours of the up step and down step. $T_{2}-1$ counts the first (non-empty) path and $T_{2}$ counts the second path. The term $u$ is in the expression because  the half length of $\omega$ is one more than the sum of the half lengths of the shorter paths. Hence we get the equation
\[T_{2}(a,u)=4u(T_{2}-1)T_{2}+u(2a+2)T_{2}+1.\]
Solving the quadratic gives 
\[T_{2}(a,u)=\frac{1+2u-2au\pm\sqrt{1-12u+4u^{2}-4au+4a^{2}u^{2}-8au^{2}}}{8u}.\]
Note that since $T_{2}(a,0)=1$, we must use the negative square root, so
\[T_{2}(a,u)=\frac{1+2u-2au-\sqrt{1-12u+4u^{2}-4au+4a^{2}u^{2}-8au^{2}}}{8u}.\]

We now calculate $T_{1}(a,u,y)$ in a similar way.
Each multicoloured Dyck path $\omega$, where each step from height 0 is red, satisfies exactly one of the following:
\begin{itemize}
\item $\omega$ is empty,
\item $\omega$ is made up of a peak followed by a (possibly empty) multicoloured Dyck path,
\item $\omega$ is made up of a red up-step, followed by a non-empty multicoloured Dyck path, followed by a down-step, followed by another (possibly empty) multicoloured Dyck path where each step from height 0 is red.
\end{itemize}
The contribution to  $T_{1}$ from the first case is simply $y$.

In the second case, the two possible peaks are counted by $uay+uy$, so the contribution to $T_{1}$ from this case is $(uay+uy)T_{1}$.

In the third case, the half length of $\omega$ is one more than the sum of the half lengths of the shorter paths, the number of vertices at height 0 is one more than the number of vertices at height 0 in the second path. Hence the contribution from this case is $2uy(T_{2}-1)T_{1}$. Hence, we get the equation
\[T_{1}(a,u,y)=2uy(T_{2}-1)T_{1}+(uay+uy)T_{1}+y.\]
Solving this using the formula for $T_{2}$ gives
\[T_{1}(a,u,y)=\frac{4y}{4+2yu-2yau-y+y\sqrt{1-12u+4u^{2}-4au+4a^{2}u^{2}-8au^{2}}}\]

Finally, we can calculate $T(a,u,x)$ in a similar way.
Each multicoloured Dyck path $\omega$, where each step from height 0 or 1 is red, satisfies exactly one of the following:
\begin{itemize}
\item $\omega$ is empty
\item $\omega$ is made up of a red up-step, followed by a (possibly empty) multicoloured Dyck path, where each step from height 0 is red, followed by a red down-step, followed by another (possibly empty) multicoloured Dyck path where each step from height 0 or 1 is red.
\end{itemize}
The contribution to  $T$ from the first case is simply $1$. In the second case, the half length of $\omega$ is one more than the sum of the half lengths of the shorter paths, and the number of multicoloured peaks in $\omega$ is equal to the sum of the numbers of multicoloured peaks in the two shorter paths. The number of vertices at height 1 in the long path is equal to the number of vertices at height 0 in the first short path, plus the number of vertices at height 1 in the second short path. Hence,
\[T(a,u,x)=uT_{1}(a,u,x)\cdot T(a,u,x)+1\]
Solving this using the formula for $T_{1}$ gives
\[T(a,u,x)=\frac{4+2xu-2xau-x+x\sqrt{1-12u+4u^{2}-4au+4a^{2}u^{2}-8au^{2}}}{4-2xu-2xau-x+x\sqrt{1-12u+4u^{2}-4au+4a^{2}u^{2}-8au^{2}}}\]
as required.\end{proof}

\begin{lemma} The generating function $Q_{1}(a,u,x)$ is given by the equation
\begin{equation}\label{Q1}Q_{1}(a,u,x)=\frac{2Q}{2Q-xQ+x}\end{equation}
\end{lemma}
\begin{proof}Let $U(a,u)$ be the generating function for non-empty quarter plane loops, which only touch the origin at the start and end, counted by half-length, conjugate to $u$ and number of $NW$ or $ES$ corners, conjugate to $a$. Then each loop counted by $Q$ can be written uniquely as a sequence of loops counted in $U$, so we get the equation
\[Q(a,u)=\frac{1}{1-U(a,u)}.\]
Or, equivalently,
\[U(a,u)=1-\frac{1}{Q(a,u)}.\]
Now, by reflecting a loop about the line $x=y$, we see that amongst the loops counted by $U$ with a given half-length and number of $NW$ or $ES$ corners, exactly half begin with a north step.  Hence, each loop counted by $Q_{1}$ can be written uniquely as a sequence of loops counted in $\frac{1}{2}U$. Moreover, the power of $x$ in the corresponding monomial in $Q_{1}$ is equal to the number of terms in the sequence of loops from $\frac{1}{2}U$. Therefore, we have the equation
\[Q_{1}(a,u,x)=\frac{1}{1-\frac{1}{2}U(a,u)x}\]
Hence, we can write $Q_{1}$ in terms of $Q$ as desired
\begin{align*}Q_{1}(a,u,x)&=\frac{1}{1-\frac{1}{2}(1-\frac{1}{Q})x}\\
&=\frac{2Q}{2Q-Qx+x}.\end{align*}\end{proof}

Before we can relate these generating functions to $M$ and $D,$ we will need to consider a decomposition of operation sequences, as described in the lemma below. In lemma \ref{opdecu} we will show that this decomposition is unique.

\begin{lemma}\label{opdece}Given a non-empty operation sequence $w$ it is possible to decompose $w$ as \[w=x_{1}w_{1}x_{2}w_{2}\ldots x_{2m-1}w_{2m-1}x_{2m}v,\]
where $m\in\mathbb{Z}_{>0}$, each $x_{i}\in\{I_{1},O_{1},I_{2},O_{2}\}$, such that $x_{1}x_{2}\ldots x_{2m}$ is an unbreakable operation sequence, each $w_{i}$ is a (possibly empty) tsip sub-word, and $v$ is an operation sequence.\end{lemma}

\begin{proof}
We construct the decomposition as follows: First we decompose $w$ as $w=uv$, so that the first point at which the deque is empty is immediately after $u$ is applied, in other words, $u$ is the minimal, non-empty prefix of $w$ which is also an operation sequence. Note that we may have $u=w$. Since the deque is empty before and after $v$ is applied, $v$ is an operation sequence. 

 Now, $x_{1}$ must be the first letter of $u$. Then, for each $i$, we let $w_{i}$ be the longest tsip-sub-word of $u$ which starts immediately after $x_{i}$. Note that this necessarily exists, since $w_{i}$ can always be the empty word. Then $x_{i+1}$ is the letter after $w_{i}$. We continue this until we have decomposed all of $u$ as $u=x_{1}w_{1}\ldots w_{j-1}x_{j}w_{j}$. Note that $u=x_{1}w_{1}\ldots w_{j}$ and $w_{j}$ both contain an equal number of $I$s and $O$s, so $x_{1}w_{1}\ldots x_{j}$ also contains an equal number of $I$s and $O$s. But since $u$ is minimal, $w_{j}$ must be empty. Note that by definition, each $w_{i}$ is a tsip-sub-word. So we just need to prove that $x_{1}x_{2}\ldots x_{j}$ is an unbreakable operation sequence. Since $x_{1}w_{1}\ldots w_{j-1}x_{j}$ is an operation sequence, and each $w_{i}$ is an operation sequence, it follows that $x_{1}x_{2}\ldots x_{j}$ is an operation sequence. Also, if $x_{1}x_{2}\ldots x_{i}$ is an operation sequence for any $i<j,$ then $x_{1}w_{1}\ldots x_{i}$ is also an operation sequence, which contradicts the minimality of $u$. Now we just need to show that $x_{1}\ldots x_{j}$ contains no tsip sub-words, other than itself and the empty word. Suppose that it does contain some tsip sub-word $x_{a}\ldots x_{i}$, with $i\geq a$. If $a>1$ then $w_{a-1}x_{a}w_{a}x_{a+1}\ldots x_{i}$ is also a tsip sub-word, which contradicts the maximality of the length of $w_{a-1}$. If $a=1$ and $i<j$, then $x_{a}x_{a+1}\ldots x_{i}$ is not an operation sequence, so it is certainly not a tsip sub-word. Hence $a=1$ and $i=j$, so $x_{1}x_{2}\ldots x_{j}$ contains no tsip sub-words other than possibly itself. Therefore $x_{1}x_{2}\ldots x_{j}$ is unbreakable.
\end{proof}

\begin{lemma}\label{opdectsip}Let $w$ be a non-empty operation sequence with the decomposition \[w=x_{1}w_{1}x_{2}w_{2}\ldots x_{2m-1}w_{2m-1}x_{2m}v,\]
where $m\in\mathbb{Z}_{>0}$, each $x_{i}\in\{I_{1},O_{1},I_{2},O_{2}\}$, such that $x_{1}x_{2}\ldots x_{2m}$ is an unbreakable operation sequence, each $w_{i}$ is a (possibly empty) tsip sub-word, and $v$ is an operation sequence. Then any tsip sub-word of $w$, which is not a prefix, is contained in one of the words $w_{i}$ or $v$.\end{lemma}

\begin{proof}
Let $u=x_{1}w_{1}x_{2}w_{2}\ldots x_{2m-1}w_{2m-1}x_{2m}$.
Suppose for the sake of contradiction that some tsip sub-word $w'$ of $w$ is not contained in any $w_{i}$ or $v$, and $w'$ is not a prefix of $w$. 

First we consider the case where $w'$ does not intersect with $v$. Then let $w'=sx_{i}w_{i}x_{i+1}w_{i+1}\ldots x_{j}t$ be a longer tsip sub-word of $u$, where $t$ is a prefix of $w_{j}$ and $s$ is a suffix of $w_{i-1}$. Since $t$ is a suffix of an operation sequence $w'$ and a prefix of another operation sequence, it must be an operation sequence. Moreover, since $t$ is a prefix of a tsip word, $t$ is also a tsip word. Similarly, $s$ is a tsip word. Hence, $s,w_{i},w_{i+1},\ldots,w_{j-1},t$ are all tsip words, as is $w'=sw_{i}x_{i+1}w_{i+1}\ldots w_{j-1}x_{j}t$, so $x_{i+1}x_{i+2}\ldots x_{j}$ is also a tsip word. But this is a contradiction, since $x_{1}x_{2}\ldots x_{2m}$ is unbreakable.

Now we consider the case where $w'$ intersects with $v$. Then we can write $w'=u'v'$, where $u'$ is a suffix of $u$ and $v'$ is a prefix of $v$. Since $w'$ is not a prefix of $w$, the word $u'$ is also not a prefix. Since $w'$ is not contained in $v$, the word $u'$ is non-empty. Since $u'$ is a prefix of an operation sequence and a suffix of an operation sequence, it is also an operation sequence. But then $u'$ satisfies the properties of $w'$ in the first case, a contradiction.\end{proof}

\begin{lemma}\label{opdecu}Given a non-empty operation sequence $w$ there is a unique way to decompose $w$ as \[w=x_{1}w_{1}x_{2}w_{2}\ldots x_{2m-1}w_{2m-1}x_{2m}v,\]
where $m\in\mathbb{Z}_{>0}$, each $x_{i}\in\{I_{1},O_{1},I_{2},O_{2}\}$, such that $x_{1}x_{2}\ldots x_{2m}$ is an unbreakable operation sequence, each $w_{i}$ is a (possibly empty) tsip sub-word, and $v$ is an operation sequence.\end{lemma}

\begin{proof}
Let $w=x_{1}w_{1}x_{2}w_{2}\ldots x_{2m-1}w_{2m-1}x_{2m}v$ be one such decomposition of $w$. We just need to show that this is the same as the decomposition constructed in lemma \ref{opdece}.

First we will show that $u=x_{1}w_{1}x_{2}w_{2}\ldots x_{2m-1}w_{2m-1}x_{2m}$ is the shortest operation sub-sequence starting from the start of $w$. Since each $w_{i}$ is an operation sequence, and $x_{1}\ldots x_{2m}$ is an operation sequence, $x_{1}w_{1}x_{2}w_{2}\ldots x_{2m-1}w_{2m-1}x_{2m}$ is an operation sequence. Let $u'$ be a non-empty prefix of $u$ such that $u'\neq u$. Then $u'=x_{1}w_{1}\ldots x_{i-1}w_{i-1}x_{i}t$ for some $i<2m$, where $t$ is a prefix of $w_{i}$. Then, since $x_{1}\ldots x_{2m}$ is unbreakable, there are strictly more $I$s than $O$s in $x_{1}\ldots x_{i}$. Also, each $w_{k}$ contains an equal number of $I$s and $O$s and $t$ contains at least as many $I$s as $O$s. Hence $u'$ contains more $I$s than $O$s, so it is not an operation sequence. Therefore, $u$ is the shortest non-empty prefix of $w$, which is also an operation sequence.

Finally, by the previous lemma, every tsip sub-word of $w$ is either a prefix of $w$ or is contained in one of the words $w_{i}$ or $v$. Hence, each $w_{i}$ is the tsip sub-word of maximal length with that starting point. Therefore, this decomposition is the same as the decomposition constructed in lemma \ref{opdece}\end{proof}

\begin{lemma}\label{canonicaldecomp}  Let $w$ be a non-empty operation sequence, with decomposition $w=x_{1}w_{1}\ldots w_{2m-1}x_{2m}v$. Then $w$ is canonical if and only if the following conditions hold:
\begin{itemize}
\item The (unbreakable) operation sequence $s=x_{1}x_{2}\ldots x_{2m}$ is top happy,
\item Each tsip word $w_{i}$ is standard and outputs eagerly, 
\item $v$ is canonical,
\item If some $x_{i}x_{i+1}$ is either $I_{1}O_{2}$ or $I_{2}O_{1}$, then $w_{i}$ is non-empty.
\end{itemize}
\end{lemma}
\begin{proof}
If $w$ is canonical, then $w$ is top happy, so $s$ is top happy and $v$ is top happy. Since $w$ is standard and outputs eagerly, any sub operation sequence of $w$ is also standard and outputs eagerly; in particular this includes each $w_{i}$, as well as $v$. Hence $v$ is canonical. Finally the fourth condition follows immediately from the fact that $w$ outputs eagerly.

Now we will assume the four conditions and prove that $w$ is canonical. A sub-word of the form $IO$ can only appear in $w$ inside one of the sub-words $w_{i}$ or $v$ or as $x_{i}x_{i+1}$, where $w_{i}$ is empty. The conditions clearly make it impossible for such a sub-word to be either $I_{1}O_{2}$ or $I_{2}O_{1}$, so $w$ outputs eagerly. Now we consider $w$ to be a bi-coloured Dyck path. Since $s$ and $v$ are top happy, every step from height 0 or 1 in $w$ which comes from $s$ or $v$ is red. Any step from height 0 or 1 which comes from some $w_{i}$, must be at height 1 in $w$ and height 0 in $w_{i}$. Then this step must be red since $w_{i}$ is standard. Finally, by Lemma \ref{opdectsip}, any tsip sub-word of $w$ is either a prefix of $w$, contained in $v$ or contained in one of the words $w_{i}$. Hence, since each $w_{i}$ and $v$ are standard, and $w$ begins with $I_{1}$, the word $w$ is also standard. Therefore, $w$ is canonical.
\end{proof}

Now recall that the generating function $M(a,u,x)$ is the generating function for top happy, unbreakable operation sequences, counted by half-length, conjugate to $u$, number of appearances of a sub-word $I_{1}O_{2}$ or $I_{2}O_{1},$ conjugate to $a$, and number of times when there is only one element in the deque, conjugate to $x$.
\begin{lemma} The generating function $M$ satisfies the equation
\begin{equation}\label{M}T(a,u,x)=M\left(1+\frac{a-1}{Q},uQ^2,\frac{Q_{1}}{Q}x\right)\frac{T}{Q}+1\end{equation}
\end{lemma}
\begin{proof}We first note that given an operation sequence $w$, there is a corresponding bi-coloured Dyck path, formed by replacing each $I_{1}$ with a red up-step, each $I_{2}$ with a blue up-step, each $O_{1}$ with a red down-step, and each $O_{2}$ with a blue down-step. Then the condition that the steps from height 0 or 1 in the Dyck paths counted by $T$ are all red is equivalent to the condition that the corresponding operation sequence is top happy. Hence, we can consider $T(a,u,x)$ to be the generating function for top happy operation sequences, counted by half-length, conjugate to $u$, number of consecutive steps $I_{1}O_{2}$ or $I_{2}O_{1},$  conjugate to $a$, and number of times in the procedure that the deque contains exactly one element, conjugate to $x$.

For each non-empty operation sequence $w$ counted by $T(a,u,x)$ we consider the decomposition $w=x_{1}w_{1}x_{2}w_{2}\ldots x_{2m-1}w_{2m-1}x_{2m}v$ described in lemma \ref{opdecu}. In particular, we consider the contribution to $T$ from all words $w$ with a given unbreakable operation sequence $s=x_{1}x_{2}\ldots x_{2m}$. Since $w$ is top happy, $s$ is also top happy, so we will only consider the top happy, unbreakable operation sequences $s$, which are exactly the operation sequences counted by $M$. Let $q,r$ be the number of consecutive $I_{1}O_{2}$ or $I_{2}O_{1}$ steps and number of times in the procedure given by $s$ that the deque contains exactly one element respectively. So the contribution of $s$ to $M(a,u,x)$ is $a^{q}u^{m}x^{r}$. 

Now we calculate the contribution to $T$ from all the words of the form $w$, with $s=x_{1}x_{2}\ldots x_{2m}$ fixed. Call this the contribution of $s$ to $T$. Each $w_{i}$ which begins (and ends) at height 1 in the bi-coloured Dyck path corresponding to $w$ can be any tsip-word, whose steps from height 0 in its bi-coloured Dyck path are all red. Hence the possible words $w_{i}$ are exactly those counted by $Q_{1}$. Each $w_{i}$ which begins and ends with the deque containing more than one element can be any tsip-word, and these are enumerated by $Q$. Also, the possible words $v$ are exactly those counted by $T$. Since the half-length of $w$ is $m$ plus the sum of the half-lengths of these sub-words, the contribution of $s$ to $T(1,u,1)$ is
\[u^{m}Q_{1}(1,u,1)^{r}Q(1,u)^{2m-r-1}T(1,u,1).\]
Now, each vertex at height 1 in the bi-coloured Dyck path corresponding to $w$ occurs in one of the words $w_{i}$ counted by $Q_{1}$, or in $v$. The number of these vertices in any word $w_{i}$ counted by $Q_{1}$, is equal to the number of vertices in the bi-coloured Dyck path of $w_{i}$ at height 0, which is 1 more than the power of $x$ in the contribution of $w_{i}$ to $Q_{1}(a,u,x)$. Since all other vertices at height 1 appear in $v$, the contribution of $s$ to $T(1,u,x)$ is equal to 
\[u^{m}(xQ_{1}(1,u,x))^{r}Q(1,u)^{2m-r-1}T(1,u,x).\]
Now we just need to consider the number of sub-words $I_{1}O_{2}$ and $I_{2}O_{1}$ in $w$. If we only consider the sub-sequences which occur in one of the words $w_{i}$ or $v$, we would get the contribution
\[u^{m}(xQ_{1}(a,u,x))^{r}Q(a,u)^{2m-r-1}T(a,u,x).\]
The only other situation where one of the these sub-sequences occurs is when one of the words $w_{i}$ is empty, and the surrounding letters $x_{i}x_{i+1}$ form one of these sub-words. Note that this cannot happen at height 1, since $w$ is top happy, so it only occurs in the case where $w_{i}$ is counted by $Q$. This case is counted as 1 in $Q(a,u)$ in the equation above, but it should be counted as $a$. Moreover, this is relevant to exactly $q$ of the words $w_{i}$, exactly those where $x_{i}x_{i+1}$ is either the word $I_{1}O_{2}$ or $I_{2}O_{1}$. Hence, the contribution of $s$ to $T(a,u,x)$ is
 \[u^{m}(xQ_{1}(a,u,x))^{r}Q(a,u)^{2m-r-q-1}(Q(a,u)+a-1)^{q}T(a,u,x).\]
Therefore, any top happy, unbreakable operation sequence $s$ which contributes $a^{q}u^{m}x^{r}$ to $M(a,u,x)$, contributes \[\left(1+\frac{a-1}{Q}\right)^{q}\left(uQ^{2}\right)^{m}\left(\frac{Q_{1}}{Q}x\right)^{r}\frac{T}{Q}\] to $T(a,u,x)$, and this accounts for all of $T(a,u,x)$ except for the 1 coming from the empty word. Hence, 
\[T(a,u,x)=M\left(1+\frac{a-1}{Q},uQ^{2},\frac{Q_{1}}{Q}x\right)\frac{T}{Q}+1.\]
\end{proof}

\begin{lemma} The generating function $D$ satisfies the equation
\begin{equation}\label{R}D(t)=M\left(1-\frac{1}{P},tP^2,1\right)\frac{D}{P}+1.\end{equation}
\end{lemma}
\begin{proof}
Using Proposition \ref{dequeops}, we know that $D(t)$ is the generating function for canonical operation sequences, counted by half-length. Using Lemma \ref{canonicaldecomp}, we see that $D(t)$ is the generating function for words $w$ which decompose as $w=x_{1}w_{1}\ldots w_{2m-1}x_{2m}v$, where
\begin{itemize}
\item The (unbreakable) operation sequence $s=x_{1}x_{2}\ldots x_{2m}$ is top happy,
\item Each tsip word $w_{i}$ is standard and outputs eagerly, 
\item $v$ is canonical,
\item If some $x_{i}x_{i+1}$ is either $I_{1}O_{2}$ or $I_{2}O_{1}$, then $w_{i}$ is non-empty.
\end{itemize}

As in the proof of the previous lemma, we will consider the contribution of any given top happy, unbreakable operation sequence $s=x_{1}x_{2}\ldots x_{2m}$ to $D$, assuming that $s$ is counted by the monomial  $a^{q}u^{m}x^{r}$ in $M(a,u,x)$. For each $i$, the word $w_{i}$ can be any standard tsip word which outputs eagerly, except that it can't be empty if $x_{i}x_{i+1}$ is $I_{1}O_{2}$ or $I_{2}O_{1}$ and there are $q$ such values of $i$. Also, $v$ can be any canonical operation sequence. Recall that standard tsip words which output eagerly are counted by $P$. Therefore, the contribution of $s$ to $D(t)$ is
\[t^{m}P^{2m-q-1}(P-1)^{q}D=\left(1-\frac{1}{P}\right)^{q}(tP^2)^{m}\frac{D}{P}.\]
Hence,
\[D(t)=M\left(1-\frac{1}{P},tP^2,1\right)\frac{D}{P}+1.\]
\end{proof}

\begin{thm}\label{bigthm} Let $D(t)$ be the length generating function for permutations which are sortable by a deque, and let $P(t)$ be the length generating function for permutations which are sortable by two stacks in parallel. Then $D$ and $P$ satisfy the following two (equivalent) equations:
\begin{equation}\label{SinR}P(t)=\frac{(D-1)(D-t-1)}{2t(D-1-Dt)}\end{equation}
and
\begin{equation}\label{RinS}2D(t)=2+t+2Pt-2Pt^2-t\sqrt{1-4P+4P^2-8P^2t+4P^2t^2-4Pt}.\end{equation}\end{thm}
\begin{proof}
First we substitute \eqref{Q1} into \eqref{M}, to remove $Q_{1}$:
\[T(a,u,x)=M\left(1+\frac{a-1}{Q},uQ^2,\frac{2x}{2Q-xQ+x}\right)\frac{T}{Q}+1.\]
Now we combine this with \eqref{corS} to get
\[T\left(\frac{1}{P}-1,\frac{tP^{2}}{(1-2P)^{2}},x\right)=M\left(1-\frac{1}{P},tP^2,\frac{x}{2P-1-Px+x}\right)\frac{T}{2P-1}+1.\]
Therefore,
\[T\left(\frac{1}{P}-1,\frac{tP^{2}}{(1-2P)^{2}},2-\frac{1}{P}\right)=M\left(1-\frac{1}{P},tP^2,1\right)\frac{T}{2P-1}+1.\]
Using \eqref{R}, we can write $M$ in terms of $D$ and $P$. So
\[T\left(\frac{1}{P}-1,\frac{tP^{2}}{(1-2P)^{2}},2-\frac{1}{P}\right)=\frac{(D-1)P}{D}\cdot\frac{T}{2P-1}+1.\]
Solving for $T$ gives the relation
\begin{equation}\label{TRS}T\left(\frac{1}{P}-1,\frac{tP^{2}}{(1-2P)^{2}},2-\frac{1}{P}\right)=\frac{(2P-1)D}{DP+P-D}.\end{equation}
Finally, using the formula \eqref{T} for $T$ and rearranging gives the desired result.
\end{proof}

\section{Analysis}

The main purpose of this section is to reduce the problem of showing that the generating functions $P$ and $D$ have the same radius of convergence to a few conjectures about the generating function $Q(a,u)$ for quarter plane loops. We will also give some evidence for the stronger conjecture below.

\begin{conj}\label{conjas} For $n\in\mathbb{Z}_{\geq0}$, let $p_{n}$ be the number of permutations of size $n$ which are sortable by two stacks in parallel and let $d_{n}$  be the number of permutations of size $n$ which are sortable by a deque. We conjecture based on  theorem \ref{bigthm} that $p_{n}\sim const \cdot \mu^{n}\cdot n^{\gamma}$ and $d_{n}\sim const \cdot  \mu^{n}\cdot n^{-3/2}$ for some constants $\mu$ and $\gamma$.\end{conj}

We first list three conjectures which are needed to show the above conjectures.
The following are conjectures 10, 11 and 12 in \cite{AB15}, respectively.
\begin{conj}\label{conj10}The series $Q(a,u)$ is $(a+1)$-positive. That is, $Q$ takes the form
\[Q(a,u)=\sum_{n\geq0}u^{n}P_{n}(a+1),\]
where each polynomial $P_{n}$ has positive coefficients.\end{conj}
\begin{conj}\label{conj11}the radius of convergence $\rho_{Q}(a)$ of $Q(a,\cdot)$ is given by 
\[
 \rho_{Q}(a)=
  \begin{cases} 
      \hfill \displaystyle\frac{1}{(2+\sqrt{2+2a})^{2}},    \hfill & \text{ if $a\geq -1/2$,} \\
      \hfill \displaystyle\frac{-a}{2(a-1)^2},\hfill & \text{ if $a\in[-1,-1/2]$.} \\
  \end{cases}
\]\end{conj}
\begin{conj}\label{conj12}The series $Q_{u}(a,u)=\frac{\partial Q}{\partial u}$ is convergent at $u=\rho_{Q}(a)$ for $a\geq -1/3$.\end{conj}

Let $t_{c}$ denote the radius of convergence of the generating function $P(t)$. Theorem 23 in \cite{AB15} states that assuming the two Conjectures \ref{conj10} and \ref{conj12} are both true,
\[\frac{t}{(2-\frac{1}{P(t)})^{2}}\leq\rho_{Q}(1/P(t)-1),\]
for $0\leq t\leq t_{c}$, with equality if and only if $t=t_{c}$, and
\[P(t_{c})\leq\frac{3}{2}.\]
Moreover, the following corollary in \cite{AB15} states that assuming the last three conjectures are all true, the following equation holds:
\begin{equation}\label{consteqn}\sqrt{2P(t_{c})}= 1+\sqrt{2t_{c}P(t_{c})}.\end{equation}
Note that in \cite{AB15} the symbol $S$ is used instead of $P$ to denote the generating function for permutations sortable by two stacks in parallel.

\begin{thm} Assuming the conjectures \ref{conj10}, \ref{conj11} and \ref{conj12}, the number of permutations of size $n$ sortable by two stacks in parallel and the number of permutations sortable by a deque of size $n$ have the same exponential growth rate.\end{thm}
\begin{proof}It suffices to prove that the generating functions $P(t)$ and $D(t)$ have the same radius of convergence $t_{c}$. Since every permutation which is sortable by two stacks in parallel is also sortable by a deque, the coefficients of $D(t)$ are no smaller than the coefficients of $P(t)$. Hence the radius of convergence $t_{D}$ of $D(t)$ satisfies $t_{D}\leq t_{c}$. Therefore, it suffices to prove that $D(t)$ is convergent for $t\in[0,t_{c})$.

Since $P(t)$ is convergent for $t\in[0,t_{c})$, the series 
\[1-4P(t)+4P(t)^2-8P(t)^2t+4P(t)^2t^2-4P(t)t\]
is also convergent in this region. Since $P(t)$ has positive coefficients, it is increasing on the interval $[0,t_{c}]$. Hence,
\[1=P(0)\leq P(t)\leq P(t_{c})\leq \frac{3}{2}\]
inside this interval. Therefore, since we are assuming Conjecture \ref{conj11}, we have the inequality
\[\frac{t}{(2-\frac{1}{P(t)})^{2}}\leq\rho_{Q}(1/P(t)-1)=\frac{1}{(2+\sqrt{2/P(t)})^{2}},\]
where equality holds if and only if $t=t_c$. Taking square roots on both sides and rearranging, noting that $2(2-1/P(t))=(2+\sqrt{2/P(t)})(2-\sqrt{2/P(t)})$, we get the inequality
\[\sqrt{2P(t)}\geq 1+\sqrt{2tP(t)},\]
with equality if and only if $t=t_{c}$. Now we can remove the square roots as follows to get the inequality
\begin{align*}
0\leq&(\sqrt{2P}-1-\sqrt{2tP})(\sqrt{2P}+1-\sqrt{2tP})(\sqrt{2P}-1+\sqrt{2tP})(\sqrt{2P}+1+\sqrt{2tP})\\
=&1-4P(t)+4P(t)^2-8P(t)^2t+4P(t)^2t^2-4P(t)t,\end{align*}
with equality if and only if $t=t_{c}$. Since the series 
\[1-4P(t)+4P(t)^2-8P(t)^2t+4P(t)^2t^2-4P(t)t\]
is convergent and positive for $t\in[0,t_{c})$, the series
\[\sqrt{1-4P(t)+4P(t)^2-8P(t)^2t+4P(t)^2t^2-4P(t)t}\]
is also convergent in this domain. Hence, the series
\[2D(t)=2+t+2Pt-2Pt^2-t\sqrt{1-4P+4P^2-8P^2t+4P^2t^2-4Pt}\]
is also convergent in this domain. 
\end{proof}

For the following analysis, we will assume that 
\begin{equation}\label{conjP}P(t)=c_{0}+c_{1}(1-t/t_{c})+c_{\alpha}(1-t/t_{c})^{\alpha}+o((1-t/t_{c})^{\alpha}),\end{equation}
for some constants $c_{0}$, $c_{1}$ and $c_{\alpha}$ and $\alpha$, with $c_{\alpha}\neq0$. In the next section we will present numerical evidence that $\alpha\approx1.473$, but for the following theorem we will assume only that $2>\alpha>1$.

\begin{thm}\label{mediumthm} Assuming that $P$ takes the form given in (\ref{conjP}), and assuming the three conjectures \ref{conj10}, \ref{conj11} and \ref{conj12}, we have $D(t)=D(t_{c})+k_{D}(1-t/t_{c})^{1/2}+o((1-t/t_{c})^{1/2})$, where $k_{D}$ is given by the equation
\[k_{D}=-t_{c}\sqrt{(2c_{0})^{3/2}t_{c}-2c_{1}\sqrt{t_{c}}}.\]\end{thm}

\begin{proof}First, using \eqref{RinS}, we can rewrite \eqref{consteqn} as
\[\sqrt{2c_{0}}=1+\sqrt{2t_{c}c_{0}}.\]
It follows that
\begin{align*}0&=(1+\sqrt{2t_{c}c_{0}}-\sqrt{2c_{0}})(1+\sqrt{2t_{c}c_{0}}+\sqrt{2c_{0}})(1-\sqrt{2t_{c}c_{0}}-\sqrt{2c_{0}})(1-\sqrt{2t_{c}c_{0}}+\sqrt{2c_{0}})\\
&=1-4c_{0}+4c_{0}^2-8c_{0}^{2}t_{c}+4c_{0}^2t_{c}^2-4c_{0}t_{c}\end{align*}
Now, recall that
\[2D(t)=2+t+2Pt-2Pt^2-t\sqrt{1-4P+4P^2-8P^2t+4P^2t^2-4Pt}.\]
Using \eqref{conjP} we can expand the expression under the square root as a power series in $(1-t/t_{c})$.
\begin{align*}&1-4P+4P^2-8P^2t+4P^2t^2-4Pt\\
=&(1-4c_{0}+4c_{0}^2-8c_{0}^{2}t_{c}+4c_{0}^2t_{c}^2-4c_{0}t_{c})\\
+&(-4c_{1}+8c_{0}c_{1}-16c_{0}c_{1}t_{c}+8c_{0}c_{1}t_{c}^{2}-4c_{1}t_{c}+8c_{0}^{2}t_{c}+4c_{0}t_{c}-8c_{0}^{2}t_{c}^{2})(1-t/t_{c})\\
+&(-4c_{\alpha}+8c_{0}c_{\alpha}-16c_{0}c_{\alpha}t_{c}+8c_{0}c_{\alpha}t_{c}^{2}-4c_{\alpha}t_{c})(1-t/t_{c})^{\alpha}\\
+&o((1-t/t_{c})^{\alpha})\\
=&(-4c_{1}+8c_{0}c_{1}-16c_{0}c_{1}t_{c}+8c_{0}c_{1}t_{c}^{2}-4c_{1}t_{c}+8c_{0}^{2}t_{c}+4c_{0}t_{c}-8c_{0}^{2}t_{c}^{2})(1-t/t_{c})\\
+&(-4c_{\alpha}+8c_{0}c_{\alpha}-16c_{0}c_{\alpha}t_{c}+8c_{0}c_{\alpha}t_{c}^{2}-4c_{\alpha}t_{c})(1-t/t_{c})^{\alpha}\\
+&o((1-t/t_{c})^{\alpha})\\
=&q_{1}(1-t/t_{c})+q_{\alpha}(1-t/t_{c})^{\alpha}+o((1-t/t_{c})^{\alpha}).\end{align*}
Here $q_{1}$ and $q_{\alpha}$ are constants defined by
\[q_{1}=-4c_{1}+8c_{0}c_{1}-16c_{0}c_{1}t_{c}+8c_{0}c_{1}t_{c}^{2}-4c_{1}t_{c}+8c_{0}^{2}t_{c}+4c_{0}t_{c}-8c_{0}^{2}t_{c}^{2}\]
and
\[q_{\alpha}=-4c_{\alpha}+8c_{0}c_{\alpha}-16c_{0}c_{\alpha}t_{c}+8c_{0}c_{\alpha}t_{c}^{2}-4c_{\alpha}t_{c}.\]
Since this expression is non-negative for $t\in[0,t_{c}]$, we must have $q_{1}\geq0$.

Taking the square root of this, we get
\begin{align*}&\sqrt{1-4P+4P^2-8P^2t+4P^2t^2-4Pt}\\
=&(1-t/t_{c})^{1/2}\sqrt{q_{1}+q_{\alpha}(1-t/t_{c})^{\alpha-1}+o((1-t/t_{c})^{\alpha-1})}\\
=&\sqrt{q_{1}}(1-t/t_{c})^{1/2}+o((1-t/t_{c})^{1/2})\end{align*}
Finally, we can use this to determine the asymptotics of $D$
\begin{align*}2D(t)=&2+t+2Pt-2Pt^2-t\sqrt{1-4P+4P^2-8P^2t+4P^2t^2-4Pt}\\
=&(2+t_{c}+2c_{0}t_{c}-2c_{0}t_{c}^2)-t_{c}\sqrt{q_{1}}(1-t/t_{c})^{1/2}+o((1-t/t_{c})^{1/2})\end{align*}
Simplifying this using the equation $\sqrt{2c_{0}}=1+\sqrt{2t_{c}c_{0}}$ gives the desired expression for $k_{D}$.\end{proof}

\begin{thm} With the same assumptions as in the previous theorem, and the additional assumptions that $k_{D}>0$ and $\alpha<3/2$, we have the expansion
\[D(t)=D(t_{c})+k_{D}(1-t/t_{c})^{1/2}+k_{\alpha}(1-t/t_{c})^{\alpha-1/2}+o((1-t/t_{c})^{\alpha-1/2}),\]
where $k_{\alpha}$ is given by the equation
\[k_{\alpha}=\frac{-t_{c}^{3/2}c_{\alpha}}{\sqrt{(2c_{0})^{3/2}t_{c}-2c_{1}\sqrt{t_{c}}}}\]\end{thm}

\begin{proof} From the proof of the previous theorem we have
\begin{align*}&\sqrt{1-4P+4P^2-8P^2t+4P^2t^2-4Pt}\\
=&(1-t/t_{c})^{1/2}\sqrt{q_{1}+q_{\alpha}(1-t/t_{c})^{\alpha-1}+o((1-t/t_{c})^{\alpha-1})}\end{align*}
Since $q_{1}>0$, we can expand this as a power series in $(1-t/t_{c})$ as follows
\begin{align*}&(1-t/t_{c})^{1/2}\sqrt{q_{1}+q_{\alpha}(1-t/t_{c})^{\alpha-1}+o((1-t/t_{c})^{\alpha-1})}\\
=&(1-t/t_{c})^{1/2}\left(\sqrt{q_{1}}+\frac{q_{\alpha}}{2\sqrt{q_{1}}}(1-t/t_{c})^{\alpha-1}+o((1-t/t_{c})^{\alpha-1})\right)\\
=&\sqrt{q_{1}}(1-t/t_{c})^{1/2}+\frac{q_{\alpha}}{2\sqrt{q_{1}}}(1-t/t_{c})^{\alpha-1/2}+o((1-t/t_{c})^{\alpha-1/2})\end{align*}
Therefore, we have
\begin{align*}2D(t)=&2+t+2Pt-2Pt^2-t\sqrt{1-4P+4P^2-8P^2t+4P^2t^2-4Pt}\\
=&(2+t_{c}+2c_{0}t_{c}-2c_{0}t_{c}^2)-t_{c}\sqrt{q_{1}}(1-t/t_{c})^{1/2}+\frac{t_{c}q_{\alpha}}{2\sqrt{q_{1}}}(1-t/t_{c})^{\alpha-1/2}+o((1-t/t_{c})^{\alpha-1/2})\end{align*}
Simplifying this using the equation $\sqrt{2c_{0}}=1+\sqrt{2t_{c}c_{0}}$ gives the desired expression for $k_{\alpha}$.\end{proof}

\section{Asymptotics}

In this section we make a numerical study of various generating functions related to quarter-plane loops, tsips and deques. We use the two most common methods of series analysis, the ratio method and the method of differential approximants. Full details of these methods can be found in, for example, \cite{G89}. Both methods aim to estimate the radius of convergence (r.c.) $z_c$ and associated exponent $\theta$ of functions whose asymptotic behaviour is given by 
\begin{equation}\label{simple1}
F(z) \sim A \left ( 1 - \frac{z}{z_c} \right )^\theta, \,\,\, {\rm as} \,\, z \to z_c^-.
\end{equation}
 It follows that
 \begin{equation} \label{simple2}
f_n=[z^n]F(z) \sim \frac {An^{-\theta-1}}{\Gamma(-\theta)z_c^n}.
\end{equation}
\subsection{Ratio method.}

The ratio method, as the name implies, relies on extrapolating the ratio of successive coefficients, $r_n.$ One has
\begin{equation}\label{ratios}
r_n =\frac{f_n}{f_{n-1}} = \frac{1}{z_c}\left (1 - \frac{\theta+1}{n} + o(1/n) \right ).
\end{equation}
Clearly, extrapoltating the ratios $r_n$ against $1/n$ should, for sufficiently large $n$, give a linear plot which extrapolates to $1/z_c$ at $1/n=0.$ The gradient is $-(\theta+1)/z_c.$ So from the intercept one can estimate the radius of convergence, and from the gradient and the estimate of the radius of convergence, one can estimate the exponent $\theta.$ In favourable cases, where the singularity is precisely as given in (\ref{simple1}), the term $o(1/n)$ can be replaced by $O(1/n^2),$ and the ratio plot will usually be linear from quite low values of $n.$

If the r.c. is known, or very accurately estimated from some other method, it follows from (\ref{ratios}) that a more precise estimate of the exponent can be made from estimators 
\begin{equation}\label{gest}
\theta_n=n(1-z_c \cdot r_n)-1 + o(1),
\end{equation}
where in favourable cases the term $o(1)$ can be replaced by $O(1/n).$ 

Even if the r.c. is not known, one can obtain an estimate of the exponent independent of the r.c. by extrapolating ratios of ratios, so that
\begin{equation}\label{rratios}
t_n = \frac{r_n}{r_{n-1}} = \left ( 1 + \frac{1+\theta}{n^2}  + o \left( \frac{1}{n^2} \right )\right ).
\end{equation}
 So we can define $\theta_n,$ an estimator of the exponent $\theta,$ as 
 \begin{equation}\label{g2est}
 \theta_n=(t_n-1)n^2-1 = \theta +  o (1).
 \end{equation}
 where again, in favourable cases, the term $o(1)$ can be replaced by $O(1/n).$ 
 
 \subsection{Method of differential approximants}
 
The method of differential approximants \cite{G89} fits the known coefficients of a power series to a number (typically 10 or 12) of holonomic differential equations, and uses the critical parameters (the radius of convergence and exponent at that point) of those differential equations as estimators of the corresponding quantities for the underlying series expansion. 

More precisely, one uses the known series coefficients to find polynomials $Q_{k}(z)$ and $P(z)$ such that the power series solution $\tilde{F}(z)$ of the  holonomic differential equation

\begin{equation}
\sum_{k=0}^M Q_{k}(z)\left (z\frac{{\rm d}}{{\rm d}z}\right )^k \tilde{F}(z) = P(z)
\end{equation}
agrees with the known coefficients of the function $F(z)$ being approximated. The order $M$ of the ode we refer to as the {\em order} of the approximant.

Constructing such {\em differential approximants} (DAs) is straightforward computationally, and only involves solving a linear system of equations. Several such DAs are constructed by varying the degrees of the polynomials $Q_{k}(z)$ and $P(z),$ while still using most, or all of the known series coefficients. The singularities are given by the zeros $z_i, \,\, i=1, \ldots , N_M$ of $Q_M(z),$ where $N_M$ is the degree of $Q_M(z).$   We take as the dominant singularity that which is both closest to the origin and common to all (or almost all) the DAs.
Critical exponents $\theta_i$ follow from the indicial equation of the DA. For the simplest (and most frequent) situation where there is a single root of $Q_M(z)$ at $z_i,$  
$$ \theta_i=1-M+\frac{Q_{M-1}(z_i)}{z_iQ_M'(z_i)}. $$ Slightly more complicated expressions are known for the cases of double, triple etc. roots.
Further details of both methods can be found in \cite{G89}.

\subsection{Quarter-plane loops}
We first studied quarter-plane loops, introduced by Albert and Bousquet-M\'elou in \cite{AB15} and described in the previous sections. We generated 500 terms in the ogf from the recurrence relation given in the next sub-section. 

Using the method of differential approximants, in particular third-order approximants, we were able to estimate the critical point to an accuracy of between 6 and 20 significant digits, depending on the value of the parameter $a.$ Based on these numerical results, we conjectured the $a$-dependence of the radius of convergence, which is given as Conjecture 11 in \cite{AB15}, and is stated here immediately below Conjecture \ref{conjas}. For $a=-1$ and $a=1$ the ogf is holonomic, and we found the defining differential equation explicitly. That is to say, in those cases the differential approximants were precisely the defining odes.
It is for these two values of $a$ that 15 digit accuracy in the radius of convergence was obtained. For other values of $a,$ the precision of our estimates was lower, which incidentally is good heuristic evidence that the underlying generating function is non-holonomic.

We also confirmed the variation of the critical exponent with the parameter $a,$ as reported in \cite{AB15}. As a point of clarification, the exponents we refer to relate to the generating function, not the coefficients. For example, at $a=0,$ $u_c=(6+4\sqrt{2})^{-1},$ and we conjecture $$Q(u,0) = q_0+q_1(1-u/u_c)^{\theta}+o((1-u/u_c)^{\theta}),$$ whereas $[u^n]Q(u,0) \sim const \cdot u_c^{-n} \cdot n^{-1-\theta},$ where $\theta = \frac{\pi}{\arccos (\sqrt{2}-1)}$ and $q_0$ and $q_1$ are constants. Similarly, we find $$Q\left (u,-\frac{1}{2}\right ) = {\tilde q}_0+{\tilde q}_1(1-u/u_c)^{3/4}+o((1-u/u_c)^{3/4}),$$ where ${\tilde q}_0$ and ${\tilde q}_1$ are constants, and $u_c = 1/9,$ so that  $[u^n]Q(u,-1/2) \sim const \cdot u_c^{-n} \cdot n^{-7/4}.$ 

While we were unable to find an exact expression for the critical exponent for all values of $a,$ we estimated the exponent values for several values of $a \ge -1/2,$ and found that the exponent appears to be discontinuous at $a=-1/2,$ which is also the case for half-plane and full-plane walks. In particular, we found that the exponent increased monotonically with $a$ for $a > -1/2.$ However it appears that the exponent when $\lim_{a\to -1/2^+}$  is $1,$ but {\em at} $a=-1/2$ it is $3/4.$ Such monotonic, continuously varying exponents would preclude holonomic generating functions, except at isolated values of $a.$ Subsequently, in as yet unpublished work, Kilian Raschel told us of his conjectured result for the exponent, which is
$$\frac{\pi}{\arccos\left ( \frac{a-1}{a+1+\sqrt{2+2a}} \right )} \,\,\, {\rm for}\,\,\, a \ge 0.$$ Raschel's conjecture agrees with our numerical results for an even broader range,  notably $a > -1/2.$ As mentioned, at $a=-1/2$ our series analysis gives $0.750 \pm 0.002,$ from which we conjecture the exact value of $3/4.$ So not only do we observe the unusual phenomenon of a critical exponent varying steadily with a parameter, but also the phenomenon of a jump-discontinuity at a particular value of $a,$ in this case at $a=-1/2.$

\subsection{Deques and two stacks in parallel.}
Recall that $D(t)$ is the ogf of the number of permutations  sortable by a deque, and $P(t)$ is the corresponding ogf for the number of permutations sortable by two stacks in parallel (tsips). Then, as shown above in Theorem \ref{bigthm},

\begin{equation} \label{EP1}
P(t)=\frac{(D(t)-1)(D(t)-t-1)}{2t(D(t)-1- t  D(t))},
\end{equation}
and equivalently that
\begin{equation} \label{EP2}
D(t)=\frac{t}{2}+1+tP(t)-t^2 P(t)- \frac{t}{2}\sqrt{1-4P(t)+4P(t)^2-8tP(t)^2+4t^2P(t)^2-4tP(t) }.
\end{equation}

The series start: 
$$P(t)=1+t+2t^2+6t^3+23t^4+103t^5+513t^6+2760t^7+15741t^8+\cdots \,\,\, {\rm tsips},$$
$$D(t)=1+t+2t^2+6t^3+24t^4+116t^5+634t^6+3762t^7+23638t^8+\cdots \,\,\, {\rm deques}.$$\\
 We generated 500 terms in the tsip series from the functional equations in \cite{AB15} using the following method:

For integers $n,k,x,y$, let $s(n,k,x,y)$ denote the number of $n$ step quarter plane walks with $k$ ES or NW corners which start at $(0,0)$ and end at $(x,y)$, so that
\[Q(a,u)=\sum_{n=0}^{\infty} \sum_{k=0}^{\infty}s(n,k,0,0)a^{k}u^{n}.\]
We calculated the values $s(n,k,x,y)$ for $n\leq500$ via the recurrence relation
\begin{align*}s(n,k,x,y)=&s(n-1,k,x-1,y)+s(n-1,k,x,y-1)\\
+&s(n-1,k,x+1,y)+s(n-1,k,x,y+1)\\
+&s(n-2,k-1,x+1,y-1)-s(n-2,k,x+1,y-1)\\
+&s(n-2,k-1,x-1,y+1)-s(n-2,k,x-1,y+1).\end{align*}

Now that we have calculated coefficients of $Q(a,u)$, we can expand 
\[F(p,t)=Q\left(-p,\frac{t}{(1-p)^{2}}\right)-2p+1\] as a series in $p$ and $t$. Next we calculated the first 500 terms of the series $p(t)$ which sends $F(p(t),t)$ to 0. Note that $p$ is called $S^{\bullet}$ in \cite{AB15}. Finally we calculated the first 500 terms of $P(t)$, using the equation
\[P(t)=\frac{1}{1-p(t)}.\]
Calculating the series $P(t)$ from the terms in $Q(a,u)$ is very fast, so most of the time in this algorithm was spent calculating the coefficients of $Q$.

Using the tsip series, along with Theorem \ref{bigthm}, we then generated the first 500 terms in the deque series. 
We subjected these two long series to differential approximant \cite{G89} analysis.

For tsips, 6th order DAs show the dominant singularity to be at $t_c=0.1207524975763(2)$ with an exponent $1.47309(3),$ with another exponent with the value $1.94652(2)$ at $t_c=0.1207524975763(1).$ A third exponent at the same place with the value $4.72(4)$ is also suggested. The quoted errors reflect only the scatter in individual approximant estimates, and because of the presence of confluent singularities should be multiplied by a factor of 10 at least to be on the safe side.

With 8th order DAs, we find the dominant singularity to be at $t_c=0.1207524975763(2)$ with an exponent $1.47309(4),$ and with another exponent $1.94652(2)$ at $t_c=0.1207524975764(8).$ A third exponent at the same place with the value $4.72(4)$ is also suggested.

For the deque series, the unbiased analysis showed the dominant singularity to be at $z_c =0.12075249773(4),$ with exponents values of $0.507(4), \,\, 0.970(2),$ and $1.414(1).$ The uncertainties quoted just reflect the variability in the estimates across many DAs, and the lack of overlap between $z_c$ estimates for deques and tsips beyond the $10$th significant digit suggests that they are too optimistic. Nevertheless, 10-digit agreement gives considerable credence to the conjecture that they are indeed equal, and we assume this for the subsequent analysis. The dominant exponent is very close to $1/2$ exactly, which provides numerical support for the conjectured square-root singularity obtained in the previous section.

Another way to study the exponents for these two problems is by the ratio method. We write $[t^n]P(t) \sim const \cdot \mu^n \cdot n^{g_p}$ and $[t^n]D(t) \sim const \cdot \mu^n \cdot n^{g_d}.$ Then from the Hadamard (coefficient-by-coefficient) quotient, $$q_n=\frac{[t^n]P(t)}{[t^n]D(t)} \sim const \cdot n^\theta,$$ where $\theta = g_p-g_d.$ Then the ratios of successive coefficients $r_n$ behave as
$$r_n=\frac{q_n}{q_{n-1}} \sim  1+\frac{\theta}{n} + {\rm o}(1/n).$$ We can estimate the exponent $\theta$ by extrapolating a sequence of estimators $\{\theta_n\},$ defined by $n\cdot (r_n-1) \sim \theta_n + {\rm o}(1).$ The result is shown in Figure \ref{fig:fig1}, in which $\theta_n$ is shown extrapolated against $1/n.$ While it is difficult to extrapolate this curve, a limit in the range [-0.975-- -0.972] looks plausible. 
\begin{figure}[htbp]
   \centering
   \includegraphics[width=3.3in]{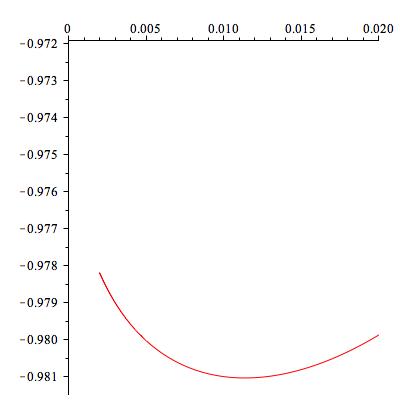} 
   \caption{Exponent estimates $\theta_n$ plotted against $1/n$. Note limit appears to be around -0.974.}
   \label{fig:fig1}
\end{figure}

We also analysed the deque series by the ratio method. From eqn (\ref{g2est}), estimators of the deque exponent $g_d$ can be found,
and these are shown extrapolated against $1/n$ in Fig \ref{fig:gn1}. The plot is quite linear and is clearly going to a value close to $-1.5.$
\begin{figure}[htbp]
   \centering
   \includegraphics[width=3.3in]{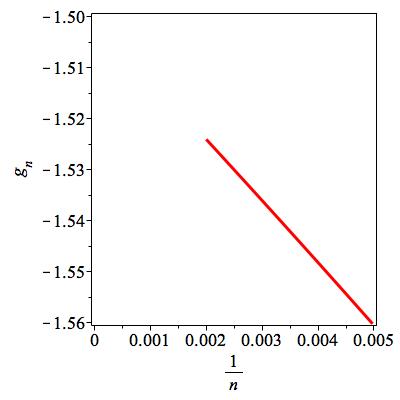} 
   \caption{Exponent estimates $g_n$ plotted against $1/n$ for deques. The limit appears to be $g \approx -1.5.$}
   \label{fig:gn1}
\end{figure}

We can refine this by calculating the exponent $g$ more precisely from the local gradient of the previous curve. This is $h_n=(g_n-g_{n-1})n(1-n),$ where $g_n$ is the $n^{th}$ estimator of $g_d.$ Then a straight line with this gradient will meet the ordinate at $g_n-\frac{h_n}{n},$ which should be a refined estimator of the exponent $g.$ This plot is shown in Fig \ref{fig:gn2}. That the curve appears to be going slightly below -1.5 is, we believe, of no consequence. We believe that if we had several thousand terms we would see this curve pass through a minimum, and increase to -1.5 exactly.
\begin{figure}[htbp]
   \centering
   \includegraphics[width=3.3in]{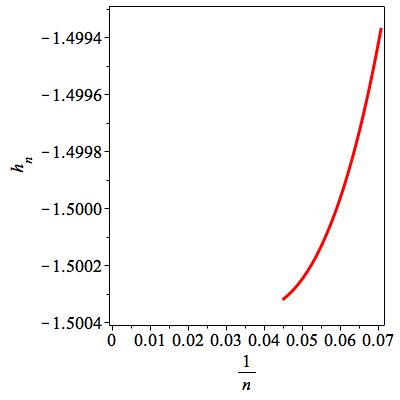} 
   \caption{Extrapolated exponent estimates $h_n$ plotted against $1/n$ for deques. The limit appears to be around $g \approx -1.500.$}
   \label{fig:gn2}
\end{figure}

Accepting the conjectured square-root singularity of the deque generating function,  numerical results thus far suggest that
\begin{equation} \label{Dz}
D(t) = D(t_c)+D_1(t_c)\sqrt{1-t/t_c}+D_2(t_c)\cdot(1-t/t_c)^{1/2+\Delta} + \cdots
\end{equation}
where $0 < \Delta \approx 0.47 < 1/2,$ as this gives the value $0.97$ observed as the sub-dominant exponent for deques.
Similarly, it appears that
\begin{equation} \label{Sz}
P(t) = P(t_c)-t_cP'(t_c)(1-t/t_c)+P_1(t_c)(1-t/t_c)^{1+\Delta}+{\rm o}((1-t/t_c)^{1+\Delta}).
\end{equation}

To see how this is consistent with (\ref{EP1}) and (\ref{EP2}), write (\ref{Dz}) as $$D(t)=D(t_c)+D_1(t_c)(1-t/t_c)^{1/2}+D_2(t_c)(1-t/t_c)^\beta + \cdots,$$
where $\beta=1/2+\Delta.$ 
Substitute into (\ref{EP1}). This gives an expression for $P(t)$ which includes terms of  $O(\sqrt{1-t/t_c})$ and $O((1-t/t_c)^{\beta}).$ Then a little algebra shows that the coefficients of both these terms  vanish if
\begin{equation}\label{eq:d0}
D(t_c)=\frac{1+t_c^{3/2}}{1-t_c}.
\end{equation}
Remarkably, this is the case, as follows from the results of  (\ref{consteqn}).

More precisely, in the proof of (\ref{mediumthm}) we have shown that
 $$ P(t_c)=\frac{1}{2(1-\sqrt{t_c})^2}$$ and$$2D(t_c)=2+t_c+2P(t_c)(t_c-t_c^2),$$ which can readily be shown to give (\ref{eq:d0}).

From eqn (\ref{Sz}) and the proof of (\ref{mediumthm}), we have
\begin{equation}\label{eq:d1}
D_1(t_c)=-2^{3/4}\cdot t_c^{3/2}\sqrt{P(t_c)^{3/2}+\sqrt{P(t_c)\cdot t_c}(1-\sqrt{t_c})\cdot P'(t_c)}.
\end{equation}

We also estimated the value of $P'(t_c)$ numerically from Pad\'e approximants to the series for $P(t)-P(t_c),$ evaluated at $t=t_c.$
Numerical estimates of $P'(t_c)$ give $D_1(t_c) \approx -0.0540$ from  (\ref{eq:d1}), which is precisely equal to direct estimates of $d_1(t_c)$ obtained from our numerical analysis of the deque series. There we formed the series for $$\frac{D(t)-D(t_c)}{\sqrt{1-t/t_c}}$$ and evaluated the approximants at $t=t_c.$ 
Furthermore, using this amplitude value $D_1(t_c),$ and subtracting $D(t_c)+D_1(t_c)\sqrt{1-t/t_c}$ from the deque series $D(t)$ gives a remainder series that behaves as $const \cdot (1-t/t_c)^\theta,$ where $\theta \approx 0.973.$ So both ratio and differential approximant analyses clearly identify this confluent exponent.

Similarly, we analysed the series for two stacks in parallel by the ratio method. The exponent estimators $g_n,$ given by eqn (\ref{gest}), are plotted against $1/n$ in Fig.  \ref{fig:2gn1}.
The plot is visually quite straight and is clearly going to a limit of about -2.475, as shown in Fig \ref{fig:2gn1}. 
\begin{figure}[htbp]
   \centering
   \includegraphics[width=3.3in]{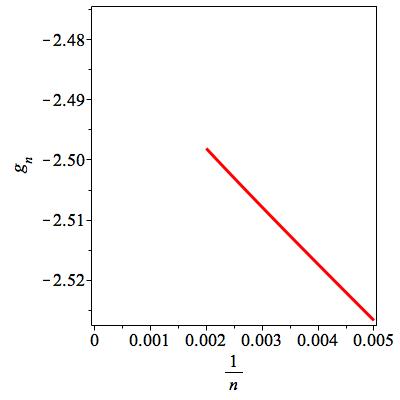} 
   \caption{Exponent estimates $g_n$ plotted against $1/n$ for two stacks in parallel. The limit appears to be about $g \approx -2.475.$}
   \label{fig:2gn1}
\end{figure}

As with the deque series, we can refine this exponent estimate by calculating the exponent $g$ more precisely from the local gradient of the previous curve. The refined estimate of the exponent $g$ is shown in  Fig \ref{fig:2gn2}, and from that plot a limit around -2.474 seems quite plausible.
\begin{figure}[htbp]
   \centering
   \includegraphics[width=3.3in]{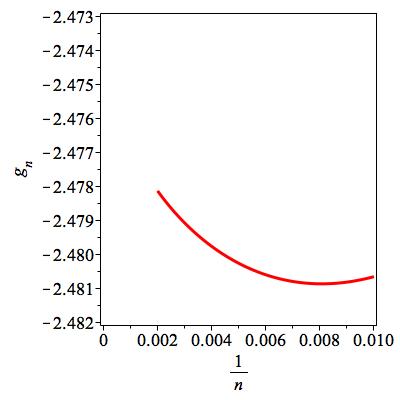} 
   \caption{Extrapolated exponent estimates $g_n$ plotted against $1/n$ for two stacks in parallel. The limit appears to be around $g \approx -2.474.$}
   \label{fig:2gn2}
\end{figure}

So the ratio plots support the conclusion that the exponent for the deque generating function is very close to $0.5,$ corresponding to a square-root singularity, as conditionally proved in Theorem \ref{mediumthm}.
The first confluent exponent for deques has the approximate value $0.973,$ and these two exponents add to give the observed value $1.473$ for the exponent of the ogf for tsips, as can be seen from  (\ref{EP1}). A direct ratio analysis of the ogf for tsips also gives an exponent value around $1.472-1.474.$

Finally, we calculate the amplitudes, or pre-multiplying constants $\kappa_d$ and $\kappa_p,$ where
$$[t^n]D(t) =d_n \sim \kappa_d \cdot t_c^{-n} \cdot n^{-3/2}, \,\,\, {\rm and} \,\,\, [t^n]P(z)=p_n \sim \kappa_p \cdot t_c^{-n} \cdot n^{-2.473}.$$ We do this by forming the sequences $d_n \cdot t_c^n \cdot n^{3/2} \sim \kappa_d + {\rm o}(1)$
and $p_n \cdot t_c^n \cdot n^{2.473} \sim \kappa_p + {\rm o}(1),$ and extrapolating these against $1/n.$ In this way we estimate $\kappa_d = 0.01524 \pm 0.0005,$ and $\kappa_p=0.08025 \pm 0.0010.$

Moving now from the arena of careful numerical work to that of wild speculation, in some unpublished work we have studied the behaviour of 421-3 pattern-avoiding permutations. We identified the critical exponent numerically, as $$\frac{2}{3} \left ( 1+ \frac{2\pi}{3\sqrt{3}} \right ) = 1.472799717437\ldots .$$ We were struck by the fact that this is tantalisingly close to the observed exponent of two stacks in parallel, though there is no obvious reason that the two problems should be connected, except that they both involve pattern-avoiding permutations.

\section{Conclusion}
For quarter-plane loops we have provided numerical support for the variation of the critical point with corner-parameter $a,$ as conjectured in \cite{AB15}. We also estimated the value of the critical exponent for a variety of values of $a \ge -1/2.$ Our numerical exponent values agree with the conjectured formula provided, in private correspondence, by Kilian Raschel. There appears to be an exponent discontinuity at $a=-1/2.$

Our principal result is Theorem \ref{bigthm}, giving the solution of the deque generating function in terms of that for tsips. Other conclusions are subject to the validity of certain conjectures. These include the result that the critical point for both the deque and tsip generating functions are equal. We also provide compelling numerical evidence for this, finding estimates that agree to 10 significant digits. Subject to these conjectures, we prove in Theorem \ref{mediumthm} that the deque generating function has a square-root singularity.

From the solutions of the deque and tsip generating functions, we produced 500 terms of the generating functions, and subjected these to careful numerical analysis. The asymptotic form of the generating functions was found to be, for deques,
\begin{equation} \label{Dzc}
D(t) \sim D(t_c)+D_1(t_c)\sqrt{1-t/t_c}+D_2(t_c)\cdot(1-t/t_c)^{1/2+\Delta} + \cdots
\end{equation}
where $ \Delta \approx 0.473,$ $t_c=0.1207524977,$ $d_0(t_c)=\frac{1+t_c^{3/2}}{1-t_c}\approx 1.185059767,$ and  $d_1(t_c) \approx -0.0543.$ For tsips we found
\begin{equation} \label{Szc}
P(t) \sim P(t_c)+-t_c P'(t_c)(1-t/t_c)+P_1(t_c)(1-t/t_c)^{1+\Delta}+P_2(t_c)(1-t/t_c)^{1+2\Delta}+ \cdots,
\end{equation}
where $P(t_c)=\frac{1}{2(1-\sqrt{t_c})^2} \approx 1.174361446,$   $P_1(t_c) \approx 0.1940,$ and $\Delta \approx 0.473.$
At the coefficient level we have $$[t^n]D(z) \approx 0.01524 \cdot t_c^{-n} \cdot n^{-3/2},$$ and
$$[t^n]P(t) \approx 0.08025 \cdot t_c^{-n} \cdot n^{-2.473}.$$

Given the simplicity of the relationship (\ref{SinRa}) between the two  generating functions $P(t)$ and $D(t)$ characterising tsips and deques respectively, there could be a much simpler proof than the one we have constructed.  It also remains a source of some frustration that we cannot establish the asymptotics rigorously, nor even unequivocally prove that the critical points of deques and tsips are identical, but hopefully our results will stimulate work in this direction.
\section*{Acknowledgments}
AJG would like to thank Mireille Bousquet-M\'elou for bringing her work with Michael Albert to his attention,  for subsequent discussions, and a careful reading of the first version of this article. We would also like to thank Kilian Raschel for telling us of his exponent conjecture, and permitting us to publish it. We are grateful to Jay Pantone for a careful reading of the manuscript, and for several corrections and improvements.  AE-P would like to acknowledge the support of the ARC Centre of Excellence for Mathematics and Statistics of Complex Systems (MASCOS).

\end{document}